\documentclass[11pt, reqno]{amsart}
\usepackage{amsfonts,amssymb,amscd,amstext,mathrsfs,color}
\usepackage[dvipsnames]{xcolor}
\usepackage{graphicx}
\usepackage{tikz}
\usepackage{hyperref}
\usepackage{dcpic,pictex} 
\usepackage[all]{xy}

\numberwithin{equation}{section}

\newtheorem{theorem}{Theorem}[section]
\newtheorem{corollary}[theorem]{Corollary}
\newtheorem{lemma}[theorem]{Lemma}
\newtheorem{proposition}[theorem]{Proposition}

\newtheorem{question}[theorem]{Question}

\theoremstyle{definition}
\newtheorem{remark}[theorem]{Remark}

\newtheorem{definition}[theorem]{Definition}
\newtheorem{example}[theorem]{Example}

\newcommand{\C}{\mathbb{C}}
\newcommand{\D}{\mathbb{D}}
\newcommand{\B}{\mathbb{B}}
\newcommand{\N}{\mathbb{N}}

\newcommand{\R}{\mathbb{R}}
\newcommand{\Q}{\mathbb{Q}}

\renewcommand{\H}{\mathbb{H}}

\begin{document}

\title{Backward dynamics of non-expanding maps in Gromov hyperbolic metric spaces}

\author{Leandro Arosio$^1$, Matteo Fiacchi$^2$,   Lorenzo Guerini, and Anders Karlsson$^3$}
 \thanks{${}^1$ Partially supported by MIUR
 Excellence Department Project awarded to the Department of
 Mathematics, University of Rome Tor Vergata, CUP E83C18000100006. }
 
  \thanks{${}^2$ Partially supported by Progetto FIRB - Bando 2012 ''Futuro in ricerca'' prot. RBFR12W1AQ002 ''Geometria Differenziale e Teoria Geometrica delle Funzioni'', by the European Union (ERC Advanced grant HPDR, 101053085 to Franc Forstneri\v c) and the research program P1-0291 from ARIS, Republic of Slovenia.}
  
  \thanks{${}^3$ Partially supported by Swiss NSF grants 200021-212864, 200020-200400, and Swedish Research Council grant 104651320.}

\maketitle
\begin{abstract}
We study the interplay between the backward dynamics of  a non-expanding self-map $f$ of a proper geodesic Gromov hyperbolic metric space $X$ and the boundary regular fixed points of $f$  in the Gromov boundary as defined in \cite{AFGG}.
 To do so, we introduce the notion of \textit{stable dilation} at a boundary regular fixed point of the Gromov boundary, whose value is related to the dynamical behaviour of the fixed point. 
 This theory applies in particular to holomorphic self-maps of bounded domains $\Omega\subset\subset  \C^q$, where $\Omega$ is either strongly pseudoconvex, convex  finite type, or pseudoconvex finite  type with $q=2$, and solves several open problems from the literature. We extend results of holomorphic self-maps of the disc $\D\subset \C$ obtained by Bracci and Poggi-Corradini in \cite{Br, PC1, PoggiCorradini}. In particular, with our geometric approach we are able to answer a question, open even for the unit   ball $\B^q\subset \C^q$ (see \cite{AbRabackERRATA,Ostapyuk}), namely that for holomorphic parabolic self-maps any escaping backward orbit with bounded step always converges to a point in the  boundary.
\end{abstract}
\tableofcontents

\section{Introduction}
The study of  dynamics of a holomorphic self-map of the unit disk $f\colon\D\to \D$ goes back to  Julia \cite{Julia} who remarked that the case where $f$ has a fixed point $p$ could be easily treated using the Schwarz lemma: either every forward  orbit converges to $p$ or $f$ is an elliptic automorphism. The situation when $f$ has no fixed point is more interesting and was described by Wolff and Denjoy in \cite{D,W}. Their celebrated theorem shows that  if $f$ has no fixed point, then there exists a  point $\zeta\in \partial \D$, called the \textit{Denjoy--Wolff point} of $f$, such that every forward orbit converges to $\zeta$. 
Such result was later generalized to various bounded domains in several complex variables: in particular, Herv\'e \cite{He} proved it for the unit ball $\B^q\subset \C^q$, and Abate proved it for bounded strongly pseudoconvex domains: 
\begin{theorem}\label{wdabate}\cite{Ab1988}
Let $f\colon \Omega\to \Omega$ be a holomorphic self-map of a bounded strongly pseudoconvex domain $\Omega\subset \C^q$.  If $f$ admits an escaping orbit,  then
 there exists a  point $\zeta\in \partial \Omega$ such that every forward orbit converges to $\zeta$.
\end{theorem}

We call an orbit \textit{escaping} if it  eventually exits any given compact set in $X$.
 To prove Theorem \ref{wdabate}, Abate exploited the non-expansivity of holomorphic maps w.r.t. the Kobayashi distance $k_\Omega$ and  proved a  ``Wolff lemma'', showing that a forward orbit is contained in a subset  (a \textit{big horosphere}) whose closure intersects the boundary in only one point. Notice  that once one proves that a forward orbit converges to a point in the boundary, it is immediate to show that every  orbit has to converge to the same point, since the Kobayashi distance between $f^n(z)$ and $f^n(w)$ is bounded by $k_\Omega(z,w).$

The metric nature of Denjoy--Wolff theorems became even more apparent with the work of Beardon \cite{B97}, who proved (with a technical assumption)
    a  Denjoy--Wolff theorem  for  non-expanding self-maps of a proper metric space $X$ with a Hausdorff compactification $\overline{X}$ which satisfies a hyperbolicity condition at the boundary called Axiom II. The full Denjoy--Wolff theorem in this setting was later proved by the fourth-named author  \cite{Kar}. The proof is again based on  a Wolff lemma\footnote{In the same paper \cite{Kar} a different proof   is also given, not based on a Wolff lemma.},  and on a result by Calka \cite{Calka} which asserts that the orbits of a non-expanding self-map of a proper metric space are either all relatively compact, or all escaping.
    
   If the metric space is Gromov hyperbolic, then its Gromov compactification satisfies Axiom  II, and thus one has the following result.
\begin{theorem}\label{Kar}\cite{Kar}
Let $f\colon X\to X$ be a non-expanding self-map of a proper Gromov hyperbolic metric space $X$. Then either 
\begin{itemize}
\item[i)]  every forward orbit of  $f$ is relatively compact, or
\item[ii)] there exists a  point $\zeta$ in the Gromov boundary $\partial_GX$ such that every forward orbit converges to $\zeta$.
\end{itemize}
\end{theorem}
Since by Balogh-Bonk \cite {BaBo} strongly pseudoconvex domains of $\C^q$ endowed with the Kobayashi distance are Gromov hyperbolic, and since the Gromov compactification is equivalent to the Euclidean compactification, this generalizes Abate's result. It also, together with all the results below of the present paper, applies to a very different setting, namely to homogeneous, order-preserving maps of proper convex cones in non-linear Perron-Frobenius theory \cite{LN}. Such maps induce non-expanding maps in Hilbert's metric on a cross-section of the cone. This cross-section is Gromov hyperbolic for example when it is strictly convex and $C^2$-smooth \cite{KN, Benoist}.   
For context let us also remark that the literature on non-expanding maps is vast, see e.g. \cite{handbook}.

Going back to the case of a  holomorphic self-map of the unit disk $f\colon  \D\to \D$, Bracci \cite{Br} and Poggi-Corradini \cite{PC1,PoggiCorradini} independently started the study of the backward dynamics of $f$ and its   interplay with  
boundary regular fixed points.
 The map  $f$ is not necessarily invertible, however it makes sense to study its backward dynamics by looking at its \textit{backward orbits}, that is sequences $(z_n)_{n\in\N}$ such that $f(z_n)=z_{n-1}$ for all $n\in \N$.
The \textit{step} of a backward orbit $(z_n)$ is defined as 
$$\sigma_1:=\sup_{n} k_{\D}(z_n,z_{n+1})\in(0,+\infty],$$
where $k$ denotes the Kobayashi distance. 

The map $f$ does not necessarily extend holomorphically (or even continuously) to the boundary of the disc, yet 
there is a notion of fixed point $\zeta$ at the boundary  which is natural in this context, and comes with a  positive real number which is a sort of derivative of $f$ at $\zeta$.
 A point $\zeta\in \partial\D$ is a \textit{boundary regular fixed point} (BRFP for short)
 if the non-tangential limit of $f$ at $\zeta$ is $\zeta$, and
  if the \textit{dilation} $\lambda_\zeta$ defined by  
 \begin{equation}\label{defdil}
 \log \lambda_\zeta:=\liminf_{z\to \zeta}k_\D(0, z)- k_\D(0, f(z))
 \end{equation}
is finite. 
The dilation $\lambda_\zeta$ at a BRFP has an interesting interpretation: by the Julia--Wolff--Carath\'eodory theorem (see e.g. \cite{AbBook}) $\lambda_\zeta$ is equal to the non-tangential limit  as $z\to \zeta$ of both the derivative $f'(z)$ and the incremental ratio $(f(z)-1)/(z-1)$.

A BRFP is
\textit{attracting} if $\lambda_\zeta<1$, \textit{indifferent}  if $\lambda_\zeta=1$, and \textit{repelling} if $\lambda_\zeta>1$.
This allows to classify holomorphic self-maps of $\D$ as follows: $f$ is \textit{ elliptic} if it has  a fixed point in $\D$. If $f$ is  not elliptic, then one can show that its Denjoy--Wolff point $\zeta$ is a BRFP, which cannot be repelling. We then say that 
$f$ is \textit{parabolic} if $\zeta$  is indifferent, and that $f$ is \textit{hyperbolic} if $\zeta$  is attracting.

The backward dynamics of $f$ is described by the  following two  results.
 The first result shows  the existence of backward orbits with bounded step converging to a given repelling BRFP.
\begin{theorem}[\cite{PC1}]\label{intro1}
Let $f\colon \D\to \D$ be a holomorphic self-map, and let $\eta$ be a repelling BRFP. Then there exists a backward orbit $(z_n)$ converging radially to $\eta$ with step $\log \lambda_\eta$.
\end{theorem}
Poggi-Corradini used Theorem \ref{intro1} to construct canonical pre-models (or Poincar\'e maps) associated with repelling BRFPs.
The proof of Poggi-Corradini of Theorem \ref{intro1} was generalized to the ball   by Ostapyuk \cite{Ostapyuk}, and to strongly convex domains by Abate--Raissy \cite{AbRaback,AbRabackERRATA},  in both cases with the additional assumption that the BRFP is isolated. A  proof without such assumption was later given in the ball in \cite{ArGu}, and in strongly convex domains in \cite{AAG}. In \cite{ArGu} and \cite{AAG}  this result was then used to develop a theory of canonical pre-models in several complex variables.

The second result can be thought of as a backward version of the Denjoy--Wolff theorem. Notice that relatively compact backward orbits  are trivial and can only exist if  the map is elliptic (see Proposition \ref{calkabackward}).
\begin{theorem}[\cite{Br,PoggiCorradini}]\label{intro2}
Let $f\colon \D\to \D$ be a holomorphic self-map, and let $(z_n)$ be a backward orbit with bounded step.
 If $(z_n)$ is not relatively compact, then it converges to a  BRFP $\eta\in \partial \D$.
Moreover, we have the following dichotomy: either
\begin{itemize}
\item[i)]    $\eta$ is repelling  with dilation satisfying $\log\lambda_\zeta\leq \sigma_1$, and  $(z_n)$ converges to $\eta$ non-tangentially, or
\item[ii)]    $\eta$ is indifferent, $f$ is parabolic and $\eta$ is its  Denjoy--Wolff point,  and  $(z_n)$ converges to $\eta$ tangentially.

\end{itemize}
\end{theorem}
Theorem \ref{intro2} was applied  by Bracci \cite{Br} to study boundary fixed points of commuting self-mapf of the disk. Theorem \ref{intro2} had partial generalizations in several variables: Ostapyuk \cite{Ostapyuk} treated the case of the ball $\B^q$ and Abate--Raissy considered strongly convex domains \cite{AbRaback,AbRabackERRATA}. In both cases convergence to a BRFP was established for hyperbolic maps and for  elliptic maps which admit a point $p$  such that every forward orbit converges to $p$, also called \textit{strongly elliptic maps}.  The question whether for a parabolic self-map every non-relatively compact backward orbit with bounded step converges to a point of the boundary 
remained open even in the ball $\B^q$, see \cite[Question 6.2.3]{Ostapyuk}, and
   \cite[Remark 2.4]{AbRabackERRATA}: ``\emph{[...] Thus the behaviour of backward orbits for parabolic self-maps is still not understood, even (as far as we know) in the unit ball of $\C^n$.} Theorem \ref{intromaintwo} below gives a positive answer to this question.

In this paper we show that, as in the forward dynamics case, the holomorphic structure does not play a relevant role in Theorems  \ref{intro1} and \ref{intro2}. Indeed,  we generalize both results to the case of a  non-expanding self-map $f\colon X\to X$  of a proper geodesic Gromov hyperbolic metric space.
 To do so, one first needs to define the concepts of  dilation and BRFP in this setting. This has been done in \cite{AFGG},  under the additional assumption that the Gromov compactification of the metric space $X$ is equivalent to the horofunction compactification.  Under this assumption the authors prove in  \cite{AFGG}  a generalization   of the classical Julia Lemma to the setting of non-expanding maps  (see Theorem \ref{julialemma} below). As a consequence it is showed that the dilation introduced by Abate for holomorphic maps of strongly convex domains is the right notion also in this context:  the dilation $\lambda_{\zeta,p}$ at a BRFP $\zeta$ is defined as
$$\log \lambda_{\zeta,p}(f):=\liminf_{x\to \zeta} d(x,p)-d(f(x),p),$$ where $p$ is a given base-point. The dilation does not depend on the chosen base-point.

 However,  if the metric space $X$ does not satisfy the assumption of equivalence of the two compactifications, then this definition of dilation  fails to detect the attracting/indifferent/repelling character of the BRFP,  as  one can construct an example of a space $X$ with a hyperbolic isometry whose dilation is strictly less than 1 at both the fixed points at the boundary (see Example \ref{exampleisometry}). 
We show that this issue disappears if one considers instead the \textit{stable dilation} $\Lambda_\zeta$ defined as
$$\log\Lambda_\zeta:=\lim_{n\to+\infty}\frac{\log \lambda_{\zeta,p}(f^n)}{n}.$$ 
The limit in the above definition exists and is finite thanks to the results  in Section \ref{sectiondilationiterates} studying the behaviour of the dilation $\lambda_{\zeta,p}$ under iteration of $f$. 
 In Section \ref{stabledilationsection} we prove that, even if the Gromov and horofunction compactifications are not equivalent, one can still prove an approximate Julia lemma (see Lemma \ref{deltaJ}), with an error depending only on the Gromov constant of the space $X$. As  consequence we will show that the stable dilation enjoys all the expected properties,   for instance if the map $f$ is non-elliptic, then its Denjoy--Wolff point is the only BRFP with stable dilation $\leq 1$.  Thus the stable dilation can be used to define attracting/indifferent/repelling BRFPs.
   Moreover we show that the stable dilation equals the dilation if the Gromov and horofunction  compactifications are equivalent. 

Before stating our two main results, we need a last definition. If $(x_n)$ is a backward orbit with bounded step, then for all $m\geq 1$ set $\sigma_m(x_n):=\lim_{n\to+\infty}d(x_{n+m},x_n).$
The \textit{backward step rate} of $(x_n)$ is then defined as $$b(x_n):=\lim_{m\to+\infty} \sigma_m(x_n)/m.$$
This number plays an important role in understanding the dynamics of $(x_n)$. Section \ref{sectionrepelling} is devoted to the proof of the first main result, generalizing Theorem \ref{intro1}.
\begin{theorem}
	\label{intromainone}
	Let $(X,d)$ be a proper geodesic Gromov hyperbolic space, and let $f\colon X\rightarrow X$ be a non-expanding map.
	Assume that $\eta\in\partial_GX$ is a repelling BRFP with stable dilation $\Lambda_\eta>1$.
	Then there exists a backward orbit $(x_n)$ converging to $\zeta$ with backward step rate $b(x_n)=\log\Lambda_{\eta}.$

\end{theorem}
It is easy to construct an example where no backward orbit  $(x_n)$ converging to $\zeta$ has step $\sigma_1(x_n)=\log\Lambda_\zeta$, (for instance the backward orbits of Example \ref{exampleisometry}). Similarly as in \cite{ArGu} and \cite{AAG} one constructs  the backward orbit $(x_n)$  as the limit of an iterative process with stopping time prescribed by a horosphere centered at the BRFP $\zeta$. The main novelty with respect to the previous proofs is the use of Gromov's four-point condition to show  that this iterative process converges.

The second main result, generalizing Theorem \ref{intro2}, is proved in Section \ref{backwardDW}. An  elliptic map is \textit{strongly elliptic} if the union of the $\omega$-limits of its forward orbits (the \textit{limit retract} of $f$) is relatively compact in $X$, otherwise it is \textit{weakly elliptic}. Notice that this definition agrees with the one previouly given for holomorphic maps.
\begin{theorem}\label{intromaintwo}
Let $(X,d)$ be a proper geodesic Gromov hyperbolic space, and let $f\colon X\rightarrow X$ be a non-expanding map, not weakly elliptic. Let $(x_n)$ be a backward orbit with bounded step.
 If $(x_n)$ is not relatively compact, then it converges to a  BRFP $\eta\in \partial_GX$.
Moreover, we have the following dichotomy: either
\begin{itemize}
\item[i)]    $b(x_n)>0$,   $\eta$ is repelling  with stable dilation satisfying $\log \Lambda_\zeta=b(x_n)$, and $(z_n)$ is a discrete quasigeodesic converging to $\eta$ inside a geodesic region, or
\item[ii)]  $b(x_n)=0$, $f$ is  parabolic, $\eta$ is the Denjoy--Wolff point of $f$,  and $(z_n)$ converges to $\eta$  avoiding an 
horosphere $\{h_{a,p}\leq c\}$ centered in a point $a$ of the horofunction boundary associated with $\eta$. 
\end{itemize}
\end{theorem}
This answers\footnote{Given a holomorphic map $f\colon \B^q\to \B^q$, if a backward orbit with bounded step converges to the Denjoy--Wolff point, then its associated canonical model is parabolic.} positively \cite[Question 9.6]{Ar}.
 It is interesting to notice that the available proofs of the Denjoy--Wolff theorem do not appear to work when applied to backward orbits.  In particular, the classical proof based the Wolff lemma cannot work since the Wolff lemma does not hold for backward orbits with bounded step, as shown by the following counterexample given by Poggi-Corradini in \cite{PoggiCorradini}: the parabolic holomorphic self-map $f(z)=\sqrt{z^2-1}$ of the upper half-plane has the backward orbit  with bounded step $(\sqrt{n+i})$ converging to the Denjoy--Wolff point $\infty$ and eventually leaving any horosphere centered at $\infty$.
  On the other side, the proof of Theorem \ref{intromaintwo} may be easily adapted to forward dynamics to give an alternative proof of the Denjoy--Wolff theorem in proper geodesic Gromov hyperbolic metric spaces, different from the two proofs in \cite{Kar}. The crucial point of the proof of Theorem \ref{intromaintwo} is that the backward step rate of $(x_n)$ is strictly positive if and only if $(x_n)$ is a discrete quasigeodesic. Hence, if $b(x_n)>0$, we can exploit  Gromov's shadowing lemma and obtain convergence of $(x_n)$  to a repelling BRFP, inside a geodesic region.   If $b(x_n)=0$, we show that $f$ cannot be strongly elliptic or hyperbolic, hence it has to be parabolic. The proof is then complete, since we show in Section \ref{stabledilationsection} that the limit points of a backward orbit with bounded step are BRFPs with stable dilation $\leq 1$. Hence  the only limit point is  the Denjoy--Wolff point of $f$.
If the map $f$ is weakly elliptic, then we cannot exclude that a backward orbit with bounded step could have limit set larger than a point and contained in the intersection of the  the Gromov boudary of $X$ with the closure of the limit retract of $f$. 

 Notice that Theorems \ref{intromainone}  and \ref{intromaintwo}  can be applied to holomorphic self-maps of bounded strongly pseudoconvex domains in $\C^q$, to smoothly bounded convex domains of finite D'Angelo  type in  $\C^q$, and to smoothly  bounded pseudoconvex domains of  finite D'Angelo type in $\C^2$. The Gromov compactification is equivalent to the Euclidean compactifications in all those cases (see respectively \cite{BaBo},\cite{Zim1} and \cite{fia}).

\section{Preliminaries}
We start reviewing some basic definitions and results which we will need in the following sections.
\subsection{Gromov hyperbolicity}

\begin{definition}
Let $\delta>0$. 
A proper geodesic metric space $(X,d)$ is $\delta$-\textit{hyperbolic} if for  every geodesic triangle, any side is contained in the $\delta$-neighborhood of the union of the two other sides. The space $(X,d)$ is  \textit{Gromov hyperbolic} if it is $\delta$-hyperbolic for some $\delta$.

\end{definition}

\begin{definition}[Gromov compactification]\label{visual}
	Let $(X,d)$ be a proper geodesic Gromov hyperbolic metric space. 
	Let $\mathscr{R}(X)$ denote the set of geodesic rays in $X$. On $\mathscr{R}(X)$, the relation 
	\[\gamma \sim_r \sigma \iff \text{$\gamma$ and $\sigma$ are asymptotic}\]
	is an equivalence relation. 
	The \textit{Gromov boundary}  of $X$ is defined as $\partial_GX:=\mathscr{R}(X)/_{\sim_r}.$
	The \textit{Gromov compactification}  of $X$ is the set $\overline X^G:=X\sqcup \partial_GX$ endowed with a  suitable compact metrizable topology (see e.g. \cite[Chapter III.H]{BH}). 

\begin{remark}\label{fourpointcondition} There is an alternative definition of Gromov hyperbolic spaces. If $(X,d)$ is a metric space, the \textit{Gromov product} of $x,y$ w.r.t. $z$ is defined as
$$(x,y)_z:=\frac{1}{2}(d(x,z)+d(y,z)-d(x,y)).$$
 We say that the space $(X,d)$ satisfies the \textit{Gromov four-point  condition} with constant $C\geq 0$ if
 $$(x,y)_w\geq \min\{(x,z)_w,(y,z)_w\}-C,\quad \forall\, w,x,y,z\in X.$$ Let $X$ be a proper geodesic metric space.    If $X$ is $\delta$-hyperbolic then it satisfies the Gromov four-point condition with $C=8\delta$. Conversely, if $X$ satisfies the Gromov  four-point condition with constant $C\geq 0$, then it is $4C$-hyperbolic (see e.g. \cite[Proposition 3.6]{CDP}).
\end{remark}
\end{definition}
\begin{definition}
Let $(X,d)$ be a proper geodesic   metric space. 
	Fix $A\geq1$, $B\geq0$. Let $I\subset \R$ a closed (possibly unbounded) interval.
	A map $\gamma\colon I\to X$ is a $(A,B)$-\textit{quasi-geodesic} if for every $s,t\geq0$
	$$A^{-1}|t-s|-B\leq d(\gamma(s),\gamma(t))\leq A|t-s|+B.$$
	If $I=\R_{\geq 0}$ (resp. $\R$) we say that $\gamma$ is a $(A,B)$-\textit{quasi-geodesic ray} (resp. \textit{line}).

	A sequence $(x_n)_{n\geq0}$ is a \textit{discrete} $(A,B)$-\textit{quasi-geodesic ray} if for every $n,m\geq0$
	$$A^{-1}|n-m|-B\leq d(x_n,x_m)\leq A|n-m|+B.$$
	Similarly one can define discrete quasi-geodesics lines $(x_n)_{n\in \mathbb{Z}}.$
By \cite[Remark 6.22]{AFGG}\label{lem_Interpol}
a  discrete $\left(A,B\right)$-quasi-geodesic ray can be interpolated with a $(A,A+B)$-quasi-geodesic ray.
\end{definition}

If the metric space $(X,d)$ is Gromov hyperbolic, then $\left(A,B\right)$-quasi-geodesics are ``shadowed'' by  actual geodesics as  the following fundamental result shows (for a proof, see e.g. \cite[Th\'eor\`eme 3.1, p. 41]{CDP}).  Denote by $d_{\mathcal{H}}$ the Hausdorff distance. 
\begin{theorem}[Gromov's shadowing lemma]\label{gromovshadowing}
Let $(X,d)$ be a proper geodesic  $\delta$-hyperbolic metric space. Let $\sigma\colon [0,+\infty)\to X$ be a $(A,B)$-quasi-geodesic ray. Then there exist a geodesic ray $\gamma \colon [0,+\infty)\to X$ such that $\gamma(0)=\sigma(0)$ and such that 
$$d_{\mathcal{H}}(\gamma,\sigma)\leq C(\delta, A, B).$$
As a consequence, if $\alpha$ and $\beta$ are two quasi-geodesic rays with the same endpoint in $\partial_GX$, then their Hausdorff distance $d_{\mathcal{H}}(\alpha,\beta)$ is bounded.
\end{theorem}

Let $f\colon X\to X$ be  a non-expanding self-map  of a proper geodesic Gromov hyperbolic metric space.
The concepts of geodesic regions/geodesic limits  were introduced in  \cite{AFGG}, generalizing classical concepts in complex analysis: Stolz regions/non-tangential limits in the disk $\D\subset \C$, and Koranyi regions/$K$-limits in the ball $\B^q\subset \C^q$ and in strongly convex domains   (see e.g. \cite{AbBook}).
The same is true for the concepts of dilation and boundary regular fixed point that will be introduced later on.

\begin{definition}[Geodesic region]\label{georegion}
	Let $(X,d)$ be a proper geodesic Gromov hyperbolic metric space. 
	Given $R>0$ and  a geodesic ray $\gamma\in \mathscr{R}(X)$, the \textit{geodesic region} $A(\gamma,R)$ is the open subset of $X$ of the form
	$$A(\gamma, R):=\{x\in X\colon d(x,\gamma)<R\}.$$
	The point $[\gamma]\in\partial_GX$ is called the \textit{vertex} of the geodesic region.
\end{definition}

\begin{definition}[Geodesic limit]
Let $(X,d)$ be a proper geodesic Gromov hyperbolic metric space and let $Y$ be a Hausdorff topological space. Let $f\colon X\rightarrow Y$ a map, and  let $\eta\in\partial_GX$, $\xi\in Y$. We say that $f$ has \textit{geodesic limit} $\xi$ at $\eta$
if for every sequence $(x_n)$ converging to $\eta$ contained in a geodesic region with vertex $\eta$, the sequence $(f(x_n))$ converges to $\xi$.

\subsection{Horofunctions}

\begin{definition}[Horofunction compactification]\label{horofunctionboundary}
	Let $(X,d)$ be a proper metric space. 
	Let $C_*(X)$ be the quotient of $C(X)$ by the subspace of constant functions. 
	Given $f\in C(X)$,  we denote its equivalence class by $\overline f\in C_*(X)$. 
	
	Consider the embedding 
	$$i_H\colon X\longrightarrow C_*(X)$$ which sends a point $x\in X$ to  the equivalence class of the function  $d_x\colon y\mapsto d(x,y).$ 
	The \textit{horofunction compactification} $ \overline{X}^H$ of $X$ is the closure of $i_H(X)$ in  $C_*(X)$. 
	The \textit{horofunction boundary} of $X$ is the (compact) set
	$$\partial_HX:=   \overline{X}^H\setminus i_H(X).$$
	Let $a \in \partial_H X$. An \textit{horofunction} centered at $a\in \partial_H X$ is an element $h\in C\left(X\right)$ satisfying $\bar{h}=a$. 
	For every $p\in X$, the unique horofunction centered at $a$ and vanishing at $p$ is denoted by $ h_{a,p}$.
	Let $c\in \R$. The level set $\{ h_{a,p}<c\}$ (or  $\{ h_{a,p}\leq c\}$)
	is called a \textit{horosphere} (or  \textit{horoball})   centered at $a$.

\end{definition}

Let $\gamma$ be a geodesic ray. The \textit{Busemann function} $B_\gamma\colon X\times X\rightarrow \R$ associated with $\gamma$ is defined as
$$B_\gamma(x,y):=\displaystyle \lim_{t\to +\infty} d(x,\gamma(t))-d(\gamma(t),y).$$
For all $y\in X$, the function $x\mapsto B_\gamma(x,y)$ is a horofunction, and its class  $\overline{B}_\gamma\in \partial_HX$ does not depend on $y\in X$.

The   following results show how the Gromov and horofunction compactifications are related
  on a proper geodesic Gromov hyperbolic metric space. Given $A\subset X$, we denote by $\overline{A}^G$ the closure of $A$ in the Gromov compactification.
\begin{proposition}[{\cite[Proposition 4.6]{WWpacific}}]
Let $(X,d)$ be a proper geodesic Gromov hyperbolic metric space. 
There exists a continuous map $\Phi:\overline X^H\rightarrow \overline X^G$ such that ${\rm id}_X=\Phi\circ i_H$, where $i_{H}\colon X\rightarrow \overline X^H$ denotes the embedding of the space $X$ into the horofunction compactification.
\end{proposition}

\begin{proposition}[{\cite[Proposition 4.4]{WWpacific}}]\label{WW}
Let $(X,d)$ be a proper geodesic Gromov hyperbolic metric space, and let $p\in X$. 
Given $a,b\in\partial_HX$, we have $\Phi(a)=\Phi(b)$ if and only if 
$$
\sup_{x\in X}|h_{a,p}(x)-h_{b,p}(x)|\le M,\qquad\text{ for some }M>0.
$$
Furthermore if the space is $\delta$-hyperbolic and $\Phi(a)=\Phi(b)$, then we can choose $M=2\delta$.
\end{proposition}

\begin{proposition}\cite{AFGG}\label{intersectsone}
Let $(X,d)$ be a proper geodesic Gromov hyperbolic metric space, and let $p\in X, a\in \partial_HX, \eta=\Phi(a)\in \partial_GX$. Then
\begin{itemize}
\item[i)]  for all $c\in \R$,  we have $\overline{\{h_{a,p}\leq  c\}}^G\cap \partial_GX=\{\eta\}$;
\item[ii)] $\bigcap_{c\in \R}{\overline{\{h_{a,p}\leq  c\}}}^G=\{\eta\}$, and thus if $(x_n)$ is a sequence in $X$ such that $h_{a,p}(x_n)\to-\infty,$ then $x_n\to \eta$;
\item[iii)] if $(x_n)$ is a sequence converging to $\eta$ inside a geodesic region, then 
$h_{a,p}(x_n)\to-\infty.$
\end{itemize}
\end{proposition}


 The main tool to generalize the classical Julia Lemma to non-expanding maps  is the following  lemma.
\begin{lemma} \cite[Lemma 6.14]{AFGG}\label{weakJ}
Let $(X,d)$ be a proper metric space and $f\colon X\to X$ a non-expanding map. Let $p\in X$. Assume that there exists a sequence $(w_n)$ in $X$ such that 
\begin{enumerate}
\item $w_n\underset{n \to \infty}{\rightarrow} a\in \partial_HX$,
\item $f(w_n)\underset{n \to \infty}{\rightarrow} b\in \partial_HX$,
\item $d(p, w_n)-d(p, f(w_n))\underset{n \to \infty}{\rightarrow} A<+\infty.$
\end{enumerate}
Then
\begin{equation}\label{horofunctionjulia}
h_{b,p}\circ f\leq h_{a,p}+A.
\end{equation}
\end{lemma}

\begin{definition}[Dilation]\label{defdilation}
Let $(X,d)$ be a proper geodesic Gromov hyperbolic metric space. 
 Let $f\colon X\to X$ be a non-expanding self-map. 
Given $\eta\in \partial_GX$, the \textit{dilation} of $f$ at $\eta$ with respect to the  base point $p\in X$ as the number $\lambda_{\eta,p}>0$ defined by
$$\log\lambda_{\eta,p}=\liminf_{z\to \eta} d(z,p)-d(f(z),p).$$ 
\end{definition}

\begin{remark}\label{changeofbasepoint}
It is easy to see that $\log\lambda_{\eta,p}>-\infty$, and that 
the condition $\lambda_{\eta,p}<+\infty$ is independent on the choice of the base point $p\in X$. Indeed
If, $p,q\in X$,  $$\log\lambda_{\eta,q}\leq \log\lambda_{\eta,p}+2d(p,q).$$
\end{remark}

\begin{proposition}\label{strangerthings}\cite[Proposition 6.15 and 6.16]{AFGG}
Let $(X,d)$ be a proper geodesic Gromov hyperbolic metric space. 
 Let $f\colon X\to X$ be a non-expanding self-map. Let $p\in X$ and   $\eta\in \partial_GX$ such that $\lambda_{\eta,p}<+\infty$. Then there exists  $\xi\in \partial_GX$ such that  $f$ has geodesic limit $\xi$ at $\eta$.
 Moreover, if $(x_n)$ is a sequence in $X$ converging to $\eta$ such that the sequence $d(x_n,p)-d(f(x_n),p)$ is bounded from above,
	then $f(x_n)$ converges to $\xi$.
 \end{proposition}
\begin{proof} We recall the proof since it will be relevant in what follows.
Let $(x_n)$ be a sequence converging to $\eta$ such that  $d(x_n,p)-d(f(x_n),p)\leq A$, with $A\in \R$.
Up to extracting subsequences we may assume that $x_n\to a\in \partial_HX$ and $f(x_n)\to b  \in \partial_HX$, where $\Phi(a)=\eta$.  Then by Lemma \ref{weakJ} we have that
$$h_{b,p}\circ f\leq h_{a,p}+A.$$ Let now $(w_n)$ be a sequence converging to $\eta$ inside a geodesic region. Then by Proposition \ref{intersectsone} iii) we have that $h_{a,p}(w_n)\to -\infty,$ and thus
$h_{b,p}(f(w_n))\to -\infty.$ This means that $(f(w_n))$ is eventually contained in a horosphere centered in $b$, and thus by  Proposition \ref{intersectsone} we have that $f(w_n)\to \Phi(b)\in \partial_GX .$ 
\end{proof}

\end{definition}\begin{definition}[Boundary regular fixed points]\label{defbrfp}
Let $(X,d)$ be a proper geodesic Gromov hyperbolic metric space. Let $f\colon X\to X$ be a non-expanding map. We say that a point $\eta\in \partial_GX$ is  a \textit{boundary regular fixed point} (BRFP for short) if
$\lambda_{\eta,p}<+\infty$ and  if $f$ has geodesic limit $\eta$ at $\eta$. 

\end{definition}

\begin{lemma}
	\label{lem:onesequenceforall}
	Let $(X,d)$ be a proper geodesic Gromov hyperbolic metric space.
	Let $f: X\rightarrow X$ be a non-expanding map, let $p\in X$ and let $\eta\in\partial_GX$ be a BRFP.
	Suppose that $(x_n)$ is a sequence in $X$ converging to $\eta$ and such that
	$$
	\limsup_{n\to\infty} d(p, x_n) - d(p, f(x_n)) = A<+\infty.
	$$
	Then for any other $q\in X$ it holds 
	$$
	\limsup_{n\to\infty} d(q, x_n) - d(q, f(x_n)) \leq A+2\delta.
	$$
	If moreover $\overline X^H$ is  equivalent to $\overline X^G$, then $$
	\limsup_{n\to\infty} d(q, x_n) - d(q, f(x_n)) \leq A.
	$$
\end{lemma}
\begin{proof}
Since $\eta$ is a BRFP, 
by Proposition \ref{strangerthings} the sequence $(f(x_n))$ converges to $\eta$. 

Let $(w_k)$ be a subsequence of $(x_n)$ such that 
$$
d(q, w_k) - d(q, f(w_k))\stackrel{k\to\infty}\longrightarrow \limsup_{n\to\infty} d(q, x_n) - d(q, f(x_n)).
$$
Up to extracting further we can assume that there exist $a,b\in \Phi^{-1}(\eta)$ such that $w_k\to a$, $f(w_k)\to b$.
We have
\begin{align*}
& \Big( d(q,w_k) - d(q, f(w_k))\Big)-\Big( d(p,w_k) - d(p, f(w_k))\Big)  \\
	&= \Big(d(q,w_k) - d(p,w_k)\Big) -\Big(d(q, f(w_k)) - d(p, f(w_k))\Big)\stackrel{k\to\infty}\longrightarrow h_{a, p}(q) - h_{b, p}(q)\leq 2\delta,
\end{align*}
where in the last inequality we used Proposition \ref{WW}.  The proof of the last statement is similar.
\end{proof}
\begin{remark}
Hence if $\eta$ is a BRFP, then  for all $p,q\in X$  $$|\log\lambda_{\eta,p}-\log\lambda_{\eta,q}|\leq 2\delta.$$
 If  $\overline X^H\simeq \overline X^G$ the previous lemma shows that the dilation $\lambda_{\eta,p}$  is independent of the base-point $p$, and  will thus be denoted $\lambda_\eta$ in what follows (cf. \cite[Proposition 6.30]{AFGG})).
\end{remark}

A direct generalization of the Julia lemma is obtained in \cite[Theorem 6.28]{AFGG} on proper geodesic Gromov hyperbolic metric spaces such that  the Gromov compactification of $X$ is equivalent to the horofunction compactification of $X$.

\begin{theorem}[Metric Julia lemma]\label{julialemma}
Let $(X,d)$ be a proper geodesic Gromov hyperbolic metric space such that $\overline X^H$ is  equivalent to $\overline X^G$. Let $f\colon X\to X$ be a non-expanding self-map. Let $\eta\in \partial_GX$ and $p\in X$ be such  that $\lambda_{\eta,p}<+\infty$. Let $\xi\in \partial_GX$ be the geodesic limit of $f$ at $\eta$.
Then 
\begin{equation}\label{carrie2}
h_{\eta,p}\circ f\leq h_{\xi,p}+\log \lambda_{\eta,p}.
\end{equation}
\end{theorem}

\subsection{Forward dynamics of elliptic maps}

A result of   Calka  shows that an interesting dynamical dichotomy holds for non-expanding self-maps of proper metric spaces: orbits are either bounded or \textit{escaping} (that  is, there are no bounded subsequences).
 \begin{theorem}\cite{Calka}
  Let $(X,d)$ a proper metric space and let $f\colon X\to X$ be a non-expanding map. If a forward orbit $(f^n(x))$ is unbounded, then $d(f^n(x),x)\to+\infty.$
 \end{theorem}
 \begin{definition}\label{ringsofpower}
 Let $f\colon X\to X$ a non-expanding self-map of a proper metric space. We say that $f$ is \textit{elliptic} if it has a bounded orbit, or equivalently, if every orbit is bounded.
 \end{definition}

\begin{definition}
Let $f\colon X\to X$ a non-expanding self-map of a proper metric space, and let $x\in X$. The  $\omega$\textit{-limit} of $f$ at $x$, denoted $\omega_f(x)$, is the limit set of the  forward orbit $(f^n(x))$, that is the set of limits points of the sequence $(f^n(x))$.
We denote 
$$\omega_f:=\bigcup_{x\in X}\omega_f(x).$$
 \end{definition}
 
The following result is classical in the context of holomorphic self-maps of taut manifolds, see Abate \cite{AbBook} (see also Bedford \cite{bedford}).  Abate's proof adapts immediately to the case of non-expanding maps of a proper metric space. A proof in the non-expanding case of points (i) and (ii) in the  below theorem is given in Lemmens-Nussbaum \cite{LN}.
 \begin{theorem}\label{isometryomegabig}
   Let $(X,d)$ a proper metric space and let $f\colon X\to X$ be an ellipic non-expanding map.
\begin{itemize}
\item[i)] There exists a subsequence of  iterates $(f^{n_k})$ converging uniformly on compact subsets to a non-expanding retraction $r\colon X\to X$ whose image is $\omega_f$.  
\item[ii)]  $f(\omega_f)=\omega_f$ and $f|_{\omega_f}$ is an isometry of $\omega_f$.
 \item[iii)] Every limit point  $h\colon X\to X$ of the sequence $(f^n)$ is of the form
 $$h=\gamma\circ r,$$
 where $\gamma$ is an isometry of $\omega_f$. 
\end{itemize}
 \end{theorem}

 \begin{remark}
 Being a retract of a Hausdorff space, $\omega_f$ is a closed subset of $X$. 
 
 \end{remark}

\begin{definition}\label{stronglyelliptic}
We call $\omega_f$ the \textit{limit retract} of the map $f$. 
An elliptic non-expanding self-map is \textit{strongly elliptic} if the limit retract is compact, and is \textit{weakly elliptic} otherwise.
\end{definition}
\begin{remark}
When $X$ is a complete hyperbolic  complex domain in $\C^q$ endowed with the Kobayashi distance and the map $f$ is holomorphic, the limit retract is a holomorphic retract.  Since every holomorphic retract is a  complex submanifold of $X$, it is  compact if and only if it is a point. Hence Definition \ref{stronglyelliptic} agrees with the one given by Abate--Raissy in \cite{AbRaback}.
\end{remark}

We end this section recalling an interesting property of images of non-expanding retracts.
 \begin{definition}
 Let $(X,d)$ be a metric space. A subset $Z\subset X$ is a \textit{non-expanding retract} if it is the image of a non-expanding retraction $\rho\colon X\to X$.
 \end{definition}

 \begin{remark}\label{embeddingofcompactifications}
 Endow a non-expanding retract $Z$ of $X$ with the  metric $d'$ induced from $d$. 
 Then the metric space $(Z,d')$ is geodesic since for all geodesic $\gamma$ in $X$ connecting points $x,y\in Z$ the composition $\rho\circ \gamma$ is a geodesic in $Z$ connecting $x$ and $y$. Moreover $Z$ is a closed subset of $X$ and thus $(Z,d')$ is proper. 
 If $X$ is Gromov hyperbolic, then also $Z$ is Gromov hyperbolic. The inclusion $i\colon (Z,d')\to (X,d)$ is an isometric embedding of proper geodesic Gromov  spaces and thus by \cite[Theorem 3.9, Chapter III.H]{BH},  it extends to a topological embedding $i\colon {\overline Z}^G\to  {\overline X}^G.$
 \end{remark}

\section{Behaviour of the dilation at a BRFP under iterates}\label{sectiondilationiterates}
In this section we study how the dilation at a BRFP changes when we iterate the map $f$. This will be 
 relevant in Section \ref{stabledilationsection} when we introduce the concept of \textit{stable dilation}.

We start showing that   an approximate Julia Lemma  holds even if   the Gromov and horofunction  compactifications are not equivalent. Moreover the ``error'' depends only on the Gromov constant $\delta$.

\begin{lemma}[$\delta$-Julia Lemma]\label{deltaJ}
 Let $(X,d)$ be a proper geodesic $\delta$-hyperbolic metric space and let $f:X\rightarrow X$ be a non-expanding map. 
 Let $\eta\in \partial_GX$ and $p\in X$ be such  that $\lambda_{\eta,p}<+\infty$. Let $\xi\in \partial_GX$ be the geodesic limit of $f$ at $\eta$.
 Then, if  $a\in \Phi^{-1}(\eta)$ and $b\in \Phi^{-1}(\xi)$, we have
$$h_{b,p}\circ f \leq h_{a,p}+\log\lambda_{\eta,p}+4\delta.$$
\end{lemma}
\begin{proof}
Let $(x_n)$ be a sequence in $X$  converging to $\eta$ and such that   $d(x_n,p)-d(f(x_n),p)\to \log\lambda_{\eta,p}$. By Proposition \ref{strangerthings} we have that $(f(x_n))$ converges to $\xi$.
Up to extracting subsequences, we may assume that $x_n\to a_0\in \Phi^{-1}(\eta)$ and $f(x_n)\to b_0\in \Phi^{-1}(\xi).$
 By Lemma \ref{weakJ} we have
$$
h_{b_0,p}\circ f \le h_{a_0,p} + \log\lambda_{\eta,p}
$$
Now if  $a\in\Phi^{-1}(\eta)$ and $b\in\Phi^{-1}(\xi)$, by Proposition \ref{WW} we have, for all $x\in X$,
$$h_{a_0,p}(x)\leq h_{a,p}(x)+2\delta$$
and
$$h_{b,p}(f(x))\leq h_{b_0,p}(f(x))+2\delta$$
which implies 
$$h_{b,p}(f(x)) \leq h_{a,p}(x)+\log\lambda_{\eta,p}+4\delta.$$
\end{proof}

With this tool in hand, we  can prove the main result of this section.
\begin{proposition}\label{brfpiterate}
Let $(X,d)$ be a proper geodesic Gromov hyperbolic metric space and 
 let $f\colon X\to X$ be a non-expanding self-map. Let $p\in X$ and   let $\eta$ be a BRFP. Then for all $n\geq 2$ the point $\eta$ is also a BRFP for the map $f^n$. Moreover, the sequence 
 $(\frac{1}{n}\log\lambda_{\eta,p}(f^n))$ is bounded.
\end{proposition}
The proof requires several preliminary results.

\begin{lemma}\label{lem:dilationgeo}
Let $(X,d)$ be a proper geodesic $\delta$-hyperbolic metric space and let $f:X\rightarrow X$ be a non-expanding map. Let $\eta\in\partial_GX$ be a BRFP and let $\gamma:[0,+\infty)\rightarrow X$ be a geodesic ray such that $\gamma(+\infty)=\eta$. Then
$$\lim_{t\to+\infty}d(\gamma(t),\gamma(0))-d(f(\gamma(t)),\gamma(0))\leq \log\lambda_{\eta,\gamma(0)}+4\delta,$$
and for all $p\in X$,
$$\limsup_{t\to+\infty}d(\gamma(t),p)-d(f(\gamma(t)),p)\leq \log\lambda_{\eta,p}+8\delta.$$
\end{lemma}
\proof
Notice that $d(\gamma(t),\gamma(0)))-d(f(\gamma(t)),\gamma(0)))$ is non-decreasing in $t$, hence its limit exists.
Let $a\in\partial_HX$ the Busemann point of $\gamma$. By the $\delta$-Julia Lemma \ref{deltaJ} we have, for all $t\geq 0$,
$$h_{a,\gamma(0)}(f(\gamma(t)))-h_{a,\gamma(0)}(\gamma(t))\leq \log\lambda_{\eta,\gamma(0)}+4\delta. $$
Since $h_{a,\gamma(0)}(f(\gamma(t)))\geq -d(f(\gamma(t)),\gamma(0))$ and $h_{a,\gamma(0)}(\gamma(t))=-t$, it follows that
$$t-d(f(\gamma(t)),\gamma(0))\leq \log\lambda_{\eta,\gamma(0)}+4\delta. $$ The result now follows from Lemma \ref{lem:onesequenceforall}.
\endproof

%

\begin{lemma}\label{geoBRFP}
Let $(X,d)$ be a proper geodesic $\delta$-hyperbolic metric space and let $f:X\rightarrow X$ be a non-expanding map. Let $\eta\in\partial_GX$ be a BRFP and let  $\gamma:[0,+\infty)\rightarrow X$ be a geodesic ray such that $\gamma(+\infty)=\eta$. Then there exists $T\geq 0$ such that $f\circ\gamma:[T,+\infty)\rightarrow X$ is a $(1,4\delta+2)$-quasi-geodesic with endpoint $\eta$.
\end{lemma}
\proof
First of all, since $f$ is non-expanding we have for each $t_1,t_2\geq0$
$$d(f(\gamma(t_1)),f(\gamma(t_2)))\leq|t_1-t_2| $$
Let $T\geq 0$ such that for each $t\geq T$ we have
$$\log\lambda_{\eta,\gamma(0)}-1\leq t-d(f(\gamma(t)),\gamma(0)))\leq\log\lambda_{\eta,\gamma(0)}+4\delta+1.$$
Now for each $T\leq t_1\leq t_2$ we have
$$d(f(\gamma(t_1)),f(\gamma(t_2)))\geq d(f(\gamma(t_2)),\gamma(0))-d(f(\gamma(t_1)),\gamma(0))\geq t_2-t_1-4\delta-2.$$
The endpoint of $f\circ\gamma$ is $\eta$ since by assumption $f$ has geodesic limit $\eta$ at $\eta$.
\endproof

\proof[Proof of Proposition \ref{brfpiterate}]
First of all notice that by Remark \ref{changeofbasepoint}, if we prove the result for a given base-point $p$, then it holds for all base-points.  Let $\gamma$ be a geodesic ray with endpoint $\eta$, and set $p:=\gamma(0).$
Denote $a_n:=\log\lambda_{\eta,p}(f^n).$
By Lemma \ref{geoBRFP} there exists $T\geq 0$ such that the curve $\sigma:=f\circ\gamma$ is a $(1,4\delta+2)$-quasi-geodesic when $t\geq T$,  and $\sigma(+\infty)=\eta$.  It follows from \cite[Lemma 5.8]{AFGG} that 
 there exists $M\geq 0$ such that $d(\gamma(t),\sigma(t))\leq M$ for all $t\geq 0$.
We show by induction that 
 for all $n\geq 1$ we have 
$$a_{n}\leq n[a_1+
8\delta+M+d(f(p),p)],$$ and that $f^n$ has geodesic limit $\eta$ at $\eta$.
This is clear for $n=1$. Assume it true for $n\geq 1$. 
We have
\begin{align*}\liminf_{t\to+\infty}d(\sigma(t),p)-d(f^n(\sigma(t)),p)&\leq\liminf_{t\to+\infty} d(\sigma(0),p)+t-d(p,f^n(\gamma(t)))+d(f^n(\sigma(t)),f^n(\gamma(t)))\\&\leq a_n+4\delta+M+d(f(p),p).
\end{align*}
Now
\begin{align*}a_{n+1}&\leq\liminf_{t\to+\infty} d(\gamma(t),p)-d(f^{n+1}(\gamma(t)),p)\\
	&=\liminf_{t\to+\infty}d(\gamma(t),p)-d(f(\gamma(t)),p)+d(f(\gamma(t)),p)-d(f^{n+1}(\gamma(t)),p)\\&\leq a_1+4\delta+a_n+4\delta+M+d(f(p),p)\\
	&\leq(n+1)[a_1+
		8\delta+M+d(f(p),p)].
	\end{align*}
By Proposition \ref{strangerthings} it follows that the map $f^{n+1}$ has a geodesic limit as $x\to \eta$.
This limit is $\eta$, since $f^{n+1}\circ \gamma= f^n\circ (f\circ \gamma)$, and $f^n$ has geodesic limit $\eta$ at $\eta$, while $f\circ \gamma$ is contained in a geodesic region.
 \endproof
 

    \begin{remark}\label{risultatoabarais}
Let $f\colon D\to D$ be a holomorphic self-map of a bounded strongly convex domain $D\subset \C^q$. Let $\eta\in \partial D$ be a BRFP, and let $\varphi\colon \D\to D$ be a complex geodesic such that $\varphi(1)=\eta$.   Then Abate--Raissy proved \cite[Lemma 3.1]{AbRaback} that $$\lim_{t\to 1-}k_D(\varphi(t), f(\varphi(t)))=|\log \lambda_\eta|.$$ The following result shows that an approximate version of this results holds for non-expanding maps of proper geodesic Gromov hyperbolic metric spaces, where again the error depends only on the Gromov constant $\delta$.
\end{remark}

\begin{proposition}\label{prop:boundfgamma}
Let $(X,d)$ be a proper geodesic Gromov hyperbolic metric space and let $f:X\rightarrow X$ be a non-expanding map.
Let $\gamma$ be a geodesic ray converging to a BRFP $\eta$. Let $p=\gamma(0)$.
Then there exists a constant $C(\delta)\geq 0$, depending only on $\delta$, such that 
$$|\log \lambda_{\eta,p}(f)|\leq \liminf_{t\to+\infty}d(\gamma(t), f(\gamma(t)))\leq \limsup_{t\to+\infty}d(\gamma(t), f(\gamma(t)))\leq |\log \lambda_{\eta,p}(f)|+C(\delta).$$
\end{proposition}
\begin{proof}
The first inequality follows immediately from the triangle inequality:
$$d(\gamma(t), f(\gamma(t)))\geq |d(\gamma(t),p)-d(f(\gamma(t)),p)|.$$
We now prove the third inequality.	By Lemma \ref{geoBRFP} there exists $T\geq 0$ such that  the curve $f\circ\gamma|_{[T,+\infty)}$ is a $(1,4\delta+2)$-quasi-geodesic with endpoint $\eta$.
	 Let $\theta$ be a geodesic ray with $\theta(0)=f(\gamma(T))$ and endpoint $\eta$. By Gromov's shadowing lemma (Theorem \ref{gromovshadowing}) there exist a constant $C'(\delta)\geq 0$ such that the curve $f\circ \gamma|_{[T,+\infty)}$ is contained in a $C'(\delta)$-neighborhood of the geodesic $\theta$.
	 By \cite[Lemma 3.3, Chapter III.H]{BH} there exist $T_1,T_2\geq 0$ such that  for all $t\geq 0$,
	 $$d(\gamma(t+T_1),\theta(t+T_2))\leq 5\delta.$$
	Set $C''(\delta):=C'(\delta)+5\delta.$ It immediately follows that there exists a constant $T_3\geq 0$ such that 
	 $$d(f(\gamma(t)),\gamma)\leq C''(\delta),\quad \forall\,t\geq T_3.$$
	 For all $t\geq T_3$ let $s_t\geq0$ be such that $d(f(\gamma(t)),\gamma(s_t))=d(f(\gamma(t)),\gamma)$. Then for all $t\geq T_3$,
	$$d(\gamma(t), f(\gamma(t)))\leq d(\gamma(t),\gamma(s_t))+d(\gamma(s_t),f(\gamma(t))\leq d(\gamma(t),\gamma(s_t))+C''(\delta).$$
	Moreover \begin{align*}d(\gamma(t),\gamma(s_t))&= |d(\gamma(t),p)- d(\gamma(s_t),p)|=  |d(\gamma(t),p)-d(p,f(\gamma(t)))+d(p,f(\gamma(t)))- d(\gamma(s_t),p)|\\
	&\leq  |d(\gamma(t),p)-d(p,f(\gamma(t)))|+|d(p,f(\gamma(t)))- d(\gamma(s_t),p)|\\
	&\leq  |d(\gamma(t),p)-d(p,f(\gamma(t)))|+ d(f(\gamma(t)),\gamma(s_t))\\
	&\leq  |d(\gamma(t),p)-d(p,f(\gamma(t)))|+ C''(\delta).\end{align*}
Thus for all $t\geq T_3$,	
	$$d(\gamma(t), f(\gamma(t)))\leq |d(\gamma(t),p)-d(p,f(\gamma(t)))|+ 2C''(\delta)$$
	letting $t\to+\infty$ we obtain, using Corollary \ref{lem:dilationgeo} ,
	$$\limsup_{t\to+\infty}d(\gamma(t), f(\gamma(t)))\leq |\log\lambda_{\eta,p}|+ 2C''(\delta)+4\delta.$$
\end{proof}

\begin{remark}
In view of Remark \ref{risultatoabarais}, it is natural to  ask whether, with notation from the previous proposition, $$\lim_{t\to+\infty}d(\gamma(t), f(\gamma(t)))= |\log \lambda_{\eta,p}(f)|,$$ assuming that
 the Gromov compactification of $X$ is equivalent to the horofunction compactification of $X$. This turns out to be false, as the following example shows. Consider the infinite cylinder $X:=\{(x,y,z)\in \R^3\colon y^2+z^2=1\}$ with the Riemannian metric inherited from the euclidean metric on $\R^3$. If $d$ is the associated distance, we have that $\overline X^H$ is equivalent to $\overline X^G$, and both boundaries consist of a point  $-\infty$ and a point $+\infty$.
Consider the isometry
 $$f(x,y,z)=(x+1,-y, -z).$$ Then $d((x,y,z), f(x,y,z))=\sqrt{1+\pi^2}$ for all $(x,y,z)\in X$, while 
$\log \lambda_{+\infty,p}(f)=-1$.
\end{remark}

Next we prove some equivalent characterizations of BRFPs.
\begin{proposition}\label{characterizationBRFP}
Let $(X,d)$ be a proper geodesic Gromov hyperbolic metric space and let $f:X\rightarrow X$ be a non-expanding map.
Let  $\eta\in\partial_GX$. The following are equivalent:
\begin{enumerate}
	\item $\eta$ is a BRFP;
	\item there exists a geodesic ray $\gamma$ with endpoint $\eta$ such that the curve $f\circ\gamma$ is a $(1,B)$-quasi-geodesic for some $B\geq 0$ with endpoint $\eta$;
	\item there exists a geodesic ray $\gamma$ with endpoint $\eta$ such that 
	$$\limsup_{t\to+\infty}d(\gamma(t),f(\gamma(t)))<+\infty;$$
	\item we have $$\liminf_{z\to\eta}d(z,f(z))<+\infty.$$
\end{enumerate}
\end{proposition}
\proof
$[(1)\Rightarrow(2)]$ follows Proposition \ref{geoBRFP}.
$[(2)\Rightarrow(3)]$ follows from \cite[Lemma 5.8]{AFGG}.
$[(3)\Rightarrow(4)]$  is trivial.
$[(4)\Rightarrow(1)]$ for all $p\in X$ we have
$$\liminf_{z\to \eta}d(z,p)-d(f(z),p)\leq \liminf_{z\to \eta}d(z,f(z))<+\infty.$$
\endproof


\subsection{The case of equivalent compactifications}
We end this section obtaining a refined version of Proposition \ref{brfpiterate} when 
 the Gromov compactification of $X$ is equivalent to the horofunction compactification of $X$. Recall that in this case the dilation  at a BRFP $\eta$ does not depend on the base-point $p$.
 \begin{proposition}\label{prop:equiv}
	Let $(X,d)$ be a proper geodesic Gromov hyperbolic metric space such that $\overline X^H$ is  equivalent to $\overline X^G$. 
	Let $f\colon X\to X$ be a non-expanding self-map and let $\eta\in \partial_GX$ be a BRFP. Then for all $n\geq 1$ we have
	$$\lambda_{\eta}(f^n)=\lambda_{\eta}(f)^n. $$
	
\end{proposition}
We need some preliminary result.
\begin{definition}
	Let $(X,d)$ be a geodesic metric space. We say that $\gamma:[0,+\infty)\rightarrow X$ is an \textit{almost geodesic} if for each $\epsilon>0$ there exists $t_\epsilon\geq0$ such that for all $t_1,t_2\geq t_\epsilon$
	$$|t_1-t_2|-\epsilon\leq d(\gamma(t_1),\gamma(t_2))\leq |t_1-t_2|.$$
\end{definition}

%
%
%

  \begin{remark}
Let $f\colon D\to D$ be a holomorphic self-map of a bounded strongly convex domain $D\subset \C^q$. Let $p\in D$, let $\eta\in \partial D$ be a BRFP, and let $\varphi\colon \D\to D$ be a complex geodesic such that $\varphi(1)=\eta$.   Then Abate proved \cite[Lemma 2.7.22]{AbBook} that $$\lim_{t\to 1-}k_D(\varphi(t),p)-k_D( f(\varphi(t)),p)=\log \lambda_\eta.$$ This result can be generalized to our setting as follows.
\end{remark}

\begin{lemma}\label{prop:JFCGeneral}
	Let $(X,d)$ be a proper geodesic Gromov hyperbolic metric space such that $\overline X^H$ is  equivalent to $\overline X^G$. 
	Let $f\colon X\to X$ be a non-expanding self-map, 
	let $\eta\in \partial_GX$ be a BRFP and let $\gamma:[0,+\infty)\rightarrow X$ be an almost geodesic with $\gamma(+\infty)=\eta$. Then
	\begin{equation}
	\label{limitgeo}
	\lim_{t\to+\infty} d(p,\gamma(t)) - d(p,f(\gamma(t))) = \log \lambda_{\eta}.
	\end{equation} 
\end{lemma}
\begin{proof}
By   Lemma \ref{lem:onesequenceforall}, we may assume $p=\gamma(0)$.
	By  definition of dilation $$\liminf_{t\to+\infty} d(p,\gamma(t)) - d(p,f(\gamma(t))) \geq \log\lambda_{\eta}.$$  Since $\eta$ is a BRFP, it follows that the curve $f(\gamma(t))$ converges to $\eta$. Hence  for all $s\geq0$ we have
	$$\lim_{t\to+\infty}d(\gamma(s),\gamma(t))-   d(\gamma(0),\gamma(t))=h_{\eta,\gamma(0)}(\gamma(s))=\lim_{t\to+\infty}d(\gamma(s),f(\gamma(t)))-   d(\gamma(0),f(\gamma(t))),$$ 
	which implies
	$$\limsup_{t\to+\infty} d(\gamma(0),\gamma(t)) - d(\gamma(0),f(\gamma(t))) =\limsup_{t\to+\infty} d(\gamma(s),\gamma(t)) - d(\gamma(s),f(\gamma(t))). $$
	
	Now, for each $\epsilon>0$ let $t_\epsilon\geq0$ be given by  the definition of almost geodesic. By Julia's Lemma (Theorem \ref{julialemma}) we have,  for all $t\geq t_\epsilon$,
	$$
	h_{\eta,\gamma(t_\epsilon)}(f(\gamma(t))) - h_{\eta,\gamma(t_\epsilon)}(\gamma(t)) \leq \log\lambda_{\eta}.
	$$
	On the other hand by the triangle inequality we have $-d(\gamma(t_\epsilon), f(\gamma(t))) \le h_{\eta,\gamma(t_\epsilon)}(f(\gamma(t)))$. Moreover, 
	$$h_{\eta,\gamma(t_\epsilon)}(\gamma(t))\leq -d(\gamma(t),\gamma(t_\epsilon))+\epsilon,$$ hence
	$$
	\limsup_{t\to+\infty} d(\gamma(0),\gamma(t)) - d(\gamma(0),f(\gamma(t)))=\limsup_{t\to\infty}d(\gamma(t_\epsilon), \gamma(t)) - d(\gamma(t_\epsilon), f(\gamma(t))) \le \log\lambda_{\eta}+\epsilon, 
	$$
	for all $\epsilon>0$.
\end{proof}

\begin{lemma}
	Let $(X,d)$ be a proper geodesic Gromov hyperbolic metric space such that $\overline X^H$ is  equivalent to $\overline X^G$. 
	Let $f\colon X\to X$ be a non-expanding self-map, let  $\eta\in \partial_GX$ be a BRFP and let $\gamma:[0,+\infty)\rightarrow X$ be an almost geodesic with $\gamma(+\infty)=\eta$. Then the curve $f\circ\gamma$ is an almost geodesic. 
\end{lemma}

\begin{proof}
	By Proposition \ref{prop:JFCGeneral} for each $s\geq0$ we have
	$$
	\lim_{t\to+\infty} d(\gamma(s),\gamma(t))- d(\gamma(s), f(\gamma(t))) = \log\lambda_\eta.
	$$
	Therefore, given $\epsilon>0$ we can choose $t_0\geq 0$ so that for every $t,s\ge t_0$  it holds that
	$$
	|d(\gamma(t_0),\gamma(t))- d(\gamma(t_0), f(\gamma(t))) - \log\lambda_\eta|\leq\epsilon/3,
	$$
	and
	$$|t-s|-\epsilon/3\leq d(\gamma(t),\gamma(s))\leq|t-s|.$$
	Assume by contradiction that there exists $t_2>t_1\ge t_0$ so that
	$$
	d(f(\gamma(t_1)), f(\gamma(t_2))) <|t_2-t_1|-\varepsilon = t_2-t_1-\varepsilon.
	$$ 
	Then
	\begin{align*}
		\log\lambda_\eta&\geq  d(\gamma(t_0), \gamma(t_2))- d(\gamma(t_0), f(\gamma(t_2)))- \epsilon/3  \\
		&\ge d(\gamma(t_0),\gamma(t_1))-d(\gamma(t_1),\gamma(t_2)) -d(\gamma(t_0), f(\gamma(t_1))) - d(f(\gamma(t_1)), f(\gamma(t_2))) - \epsilon/3 \\
		&\geq  \log\lambda_{\eta} - \epsilon/3  - d(\gamma(t_1),\gamma(t_2))- d(f(\gamma(t_1)), f(\gamma(t_2)))- \epsilon/3 \\
		&> \log\lambda_\eta + \epsilon-2\epsilon/3=\log\lambda_\eta + \epsilon/3,
	\end{align*}
	which is  a contradiction.
\end{proof}

\begin{proof}[Proof of Proposition \ref{prop:equiv}]
Let $\gamma:[0,\infty)\rightarrow X$ be a geodesic ray such that $\gamma(+\infty)=\eta$.  Then
\begin{align*}
	\lambda_{\eta}(f^n)&=\lim_{t\to+\infty} d(p,\gamma(t)) - d(p,f^n(\gamma(t)))\\&=\lim_{t\to+\infty} [d(p,\gamma(t))-d(p,f(\gamma(t))] +\dots+[d(p,f^{n-1}(\gamma(t)))-d(p,f^n(\gamma(t))]\\&=n\log\lambda_{\eta}(f).
\end{align*}
\end{proof}
%

\section{Stable dilation at a BRFP}\label{stabledilationsection}

When $\overline X^H$ is  equivalent to $\overline X^G$ the dilation $\lambda_{\eta}$  is deeply related to the dynamical behaviour of the BRFP $\eta$, see \cite[Theorem 6.32, Proposition 6.34]{AFGG}. For this reason  the following definition was given in \cite[Definition 6.31]{AFGG} (see \cite{AbBook,AbRaback} for the same definition in the case of holomorphic self-maps of the ball or, more generally, of a strongly convex domain in $\C^q$). 
\begin{definition}
Let $(X,d)$ be a proper geodesic Gromov hyperbolic metric space such that $\overline X^H$ is topologically equivalent to $\overline X^G$. Let $f\colon X\to X$ be a non-expanding self-map, and let $\eta$ be a BRFP. We say
that $\eta$ is \textit{attracting} if $\lambda_{\eta}<1$, it is \textit{indifferent} if $\lambda_{\eta}=1$, and it is
 \textit{repelling} if $\lambda_{\eta}>1$.
\end{definition}

However, the previous definition cannot be carried \textit{verbatim} to the case when  $\overline X^H$ is not  equivalent to  $\overline X^G$. Indeed, in this case the dilation at a BRFP may depend on the base-point. More surprisingly,   even when the dilation does not depend on the base-point, it  turns out  not to be the right tool to distinguish between attracting, indifferent and repelling BRFPs, as the following example shows.

\begin{example}\label{exampleisometry}
Consider $(\R,d)$ and $(\mathbb{S}^1,d')$ where $d$ is the euclidean  distance and  $d'$ is the inner distance induced by the euclidean distance. Endow $X:=\R\times \mathbb{S}^1$ with the distance
$d''((x_1,y_1),(x_2,y_2))=d(x_1,x_2)+d'(y_1,y_2).$
Then   $X$ is a proper geodesic Gromov hyperbolic space.
The Gromov boundary consists of two points $+\infty$ and $-\infty$, while the horofunction boundary is the disjoint union of two $\mathbb{S}^1$. Let $\vartheta\in [0,\pi]$ and let  $R_\vartheta\colon \mathbb{S}^1\to \mathbb{S}^1$ be the counterclockwise rotation by an angle $\vartheta$. The hyperbolic isometry $f\colon X\to X$ defined by $f(x,y)=(x+1,R_\vartheta(y))$ has Denjoy--Wolff point $+\infty$, while $-\infty$ is the Denjoy--Wolff point of $f^{-1}.$ 
It is easy to see that  $\log \lambda_{-\infty,p}$  does not depend on the base point $p$ and is equal to $1-\vartheta$. Hence depending on $\vartheta$ the dilation $\lambda_{-\infty,p}$ can be strictly larger than 1, equal to 1, or strictly smaller than 1.

\end{example}

This motivates the following definition.
\begin{definition}[Stable dilation]
Let $(X,d)$ be a proper geodesic Gromov hyperbolic metric space. 
 Let $f\colon X\to X$ be a non-expanding self-map. Let $p\in X$ and   let $\eta$ be a BRFP. We define the \textit{stable dilation} of $f$ at $\eta$ as
 $$\log\Lambda_\eta:=\lim_{n\to\infty}\frac{\log \lambda_{\eta,p}(f^n)}{n}.$$
\end{definition}
\begin{remark}
If the Gromov and horofunction compactifications are equivalent, then by Proposition \ref{prop:equiv} it follows that $\Lambda_\eta=\lambda_\eta.$
\end{remark}

	\begin{example}
		Let  $f:X\to X$ be the isometry of Example \ref{exampleisometry}.  
	 A simple computation shows that for every $x\in \R$, $t\geq 0$ and $y,y'\in \mathbb{S}^1$ we have
		$$t \leq d''((x,y),(x+t,y'))\leq t+2\pi,$$
		hence for a fixed $n\geq 1$, for all $x\in \R$ close to $-\infty$ and for all $y\in \mathbb{S}^1$ we have
		$$n-2\pi \leq d''(p,(x,y))-d''(p,f^n(x,y))\leq n+2\pi,$$
				hence $$n-2\pi\leq \log\lambda_{-\infty,p}(f^n)\leq n+2\pi $$ which implies that $\log \Lambda_{-\infty}=1$.
	\end{example}

%

We have to show that the limit  in the previous definition exists. This is the content of the next result.
\begin{theorem}\label{erdos}
The sequence $(\log\lambda_{\eta,p}(f^n))_{n\geq 1}$ is superadditive.
\end{theorem}
\begin{remark}\label{feketesuper}
By the Fekete Lemma this implies that $$\lim_{n\to\infty}\frac{\log\lambda_{\eta,p}(f^n)}{n}=\sup_{n\geq 1} \frac{\log\lambda_{\eta,p}(f^n)}{n}\in (-\infty,\infty].$$
By Proposition \ref{brfpiterate} we have that the limit of $\log\lambda_{\eta,p}(f^n)/n$ is actually finite.
\end{remark}

\begin{proof}
Fix $n,m\geq 1$. Let $(z_k)$ be a sequence in $X$ converging to $\zeta$ such that 
$$d(z_k,p)-d(f^{n+m}(z_k),p)\to \log\lambda_{\zeta,p}(f^{n+m}).$$
Denote $q:=f^m(p).$ It follows that 
$$d(z_k,p)-d(f^{n}(z_k),p)\leq d(z_k,p)-d(f^{n+m}(z_k), q)\leq d(z_k,p)-d(f^{n+m}(z_k), p)+d(p,q)$$
which is uniformly bounded from above. By Proposition \ref{brfpiterate} we know that the geodesic limit of $f^n$ as $z\to\eta$ is $\eta$, hence Proposition \ref{strangerthings} yields that $f^n(z_k)\to \eta.$
We have
\begin{align*}
\log\lambda_{\eta,p}(f^{n+m})&=\lim_{k\to+\infty}d(z_k,p)-d(f^{n+m}(z_k),p)\\
&=\lim_{k\to+\infty}d(z_k,p) -d(f^n(z_k),p)+d(f^n(z_k),p)   -d(f^{n+m}(z_k),p)\\
&\geq \liminf_{k\to+\infty}d(z_k,p) -d(f^n(z_k),p)+\liminf_{k\to+\infty}d(f^n(z_k),p)   -d(f^{n+m}(z_k),p)\\
&\geq \log\lambda_{\eta,p}(f^n)+\log\lambda_{\eta,p}(f^m).
\end{align*}

\end{proof}

 It follows immediately from Remark \ref{changeofbasepoint} that the stable dilation $\Lambda_{\eta}$ at a BRFP does not depend on the base-point $p$.  Notice also that $\Lambda_\eta(f^n)=\Lambda_\eta(f)^n$.
\begin{definition}\label{newtype}
Let $(X,d)$ be a proper geodesic Gromov hyperbolic metric space. Let $f:X\rightarrow X$ be a non-expanding map and $\eta\in\partial_GX$ a BRFP. We say that $\eta$ is 
\textit{attracting} if $\Lambda_\eta<1$, \textit{parabolic} if $\Lambda_\eta=1$, and \textit{repelling} if $\Lambda_\eta>1.$
\end{definition}
\begin{remark}
Notice that by Remark \ref{feketesuper} it is enough for a BRFP $\eta$ to have one integer $n$ such that  $\lambda_{\eta,p}(f^n)>1$ to conclude that $\eta$ is repelling.
\end{remark}

 \begin{definition}[Divergence rate]
 Let $(X,d)$ be a metric space and $f\colon X\to X$ be a non-expanding self-map. 
 Let $x\in X$, the \textit{divergence rate} (or \textit{translation length}, or \textit{escape rate}) $c(f)$ of $f$ is the limit
 \begin{equation}\label{limex}
 c(f):=\lim_{n\to\infty}\frac{d(x,f^n(x))}{n}. 
 \end{equation}
 \end{definition}
 
\begin{remark}
The sequence $(d(x,f^n(x)))$ is subadditive, indeed if $n,m\geq 0$,
$$d(x,f^{n+m}(x))\leq d(x,f^n(x))+d(f^n(x),f^{n+m}(x))\leq d(x,f^n(x))+d(x,f^m(x)).$$ 
Hence by the Fekete Lemma the limit \eqref{limex} exists and equals
$$\inf_{n\geq 1}\frac{d(x,f^n(x))}{n}.$$
Moreover, the limit \eqref{limex}  does not depend on $x\in X$, indeed for all $x,y\in X$ we have
$$|d(x,f^n(x))-d(y,f^n(y))|\leq d(x,y)+d(f^n(x),f^n(y))\leq 2d(x,y).$$
\end{remark}


\begin{proposition}\label{sandman}
Let $(X,d)$ be a proper geodesic Gromov hyperbolic metric space. Let $f\colon X\to X$ be a non-elliptic non-expanding map.
Let  $\zeta\in\partial_GX$ be its Denjoy--Wolff point.  Then $$\log\Lambda_\zeta=-c(f).$$ 
\end{proposition}
\proof

Let $a\in\Phi^{-1}(\zeta)$. By the $\delta$-Julia Lemma \ref{deltaJ} we have
$$\log\lambda_{\zeta,p}(f^n)+4\delta\geq h_{a,p}(f^n(x))-h_{a,p}(x)\geq -d(f^n(x),x),$$
hence
$$-c(f)=\lim_{n\to+\infty}\frac{-d(f^n(x),x)}{n}\leq \lim_{n\to+\infty}\frac{\log\lambda_{\zeta,p}(f^n)+4\delta}{n}=\log\Lambda_\zeta. $$
Moreover, by (1) of \cite[Proposition 6.19]{AFGG},
$$\log\Lambda_\zeta=\lim_{n\to+\infty}\frac{\log\lambda_{\zeta,p}(f^n)}{n}\leq\lim_{n\to+\infty}\frac{-c(f^n)}{n}=-c(f).$$
\endproof

  \begin{definition}
  Let $(X,d)$ be a proper geodesic Gromov hyperbolic metric space. Let $f\colon X\to X$ be a non-expanding map.
 Recall (Definition \ref{ringsofpower}) that $f$  is called elliptic if it admits a forward orbit which is not escaping (equivalently by Calka's theorem, every forward orbit is relatively compact).
 If $f$ is non-elliptic, we say that it is
 \begin{itemize}
 \item \textit{parabolic} if its Denjoy--Wolff point is indifferent, or equivalently if $c(f)=0$;
 \item  \textit{hyperbolic}  if its Denjoy--Wolff point is attracting, or equivalently if  $c(f)>0$.
 \end{itemize}
 \end{definition}
\begin{remark}
This definition generalizes both the classification of holomorphic self-maps of a bounded strongly convex domain of $\C^n$ (see e.g. \cite{AbRaback}), and the  classification of isometries of a proper geodesic Gromov hyperbolic metric space (see for instance \cite{CDP}). 
\end{remark}

We briefly describe the BRFPs of isometries. An isometry $f\colon X\to X$ of a proper geodesic Gromov hyperbolic space extends to  a homeomorphism $\tilde f\colon {\overline X}^G\to  {\overline X}^G$ 
(see \cite[Theorem 3.9, Chapter III.H]{BH}).
Clearly if $\eta\in \partial_GX$ is a BRFP, then it  is a fixed point of $\tilde f$.
 Conversely, if $\eta\in \partial_GX$ is a fixed point of $\tilde f$, then given a geodesic ray $\gamma$ with endpoint $\eta$, the curve $f\circ \gamma$ is also a geodesic ray with endpoint $\eta$, and thus by  Proposition \ref{characterizationBRFP}  the point $\eta$ is a BRFP.
By a classical result (see e.g. \cite{CDP}), if an isometry $f$ is not elliptic then either
\begin{enumerate}
\item the unique fixed point of $\tilde f$ in  $\partial_GX$  is the Denjoy--Wolff point of $f$, and in this case $f$ is parabolic; or
\item  $\tilde f$ has exactly two fixed points in  $\partial_GX$, which are the Denjoy--Wolff points of $f$ and $f^{-1}$, and in this case $f$ is hyperbolic.
\end{enumerate}
It follows that if $f$ is parabolic, then its indifferent Denjoy--Wolff point is the only BRFP.  If $f$ is hyperbolic, and $\zeta,\eta$ denote the Denjoy--Wolff points of $f$ and $f^{-1}$ respectively, then 
it follows from Proposition \ref{sandman} and from $c(f)=c(f^{-1})$ that 
$$\log\Lambda_\eta(f)=-\log\Lambda_\eta(f^{-1})=c(f^{-1})=c(f)=-\log\Lambda_\zeta(f),$$
hence $\zeta$ is a repelling BRFP for $f$ with stable dilation
 $$\Lambda_\eta(f)=\frac{1}{\Lambda_\zeta(f)}.$$

%
%
%

 \begin{lemma}\label{ellipticisometry}
 Let $f\colon X\to X$ be an elliptic isometry of a proper geodesic Gromov hyperbolic metric space.
Then all BRFPs of $f$ are indifferent.
 \end{lemma}
 \begin{proof}
Let $\eta$ be a BRFP. Fix a point $p\in X$. We have, for all $x\in X$, $n\geq 1$,
$$-d(x,f^n(x))\leq d(x,p)-d(f^n(x),p)\leq d(x,p)-d(f^n(x),f^n(p))+d(p,f^n(p))=d(p,f^n(p)).$$
Since $f$ is elliptic, the sequences $(d(p,f^n(p)))$ and $(d(x,f^n(x)))$ are bounded, and thus $\Lambda_\eta= 1$.
 \end{proof}


\begin{proposition}\label{BRFPanddilation}
 Let $f\colon X\to X$ be a non-expanding self-map of a proper geodesic Gromov hyperbolic metric space.
\begin{itemize}
\item[i)] If a BRFP $\eta$ of $f$ has stable dilation $\Lambda_\eta< 1$ then $f$ is hyperbolic and $\eta$ is its Denjoy--Wolff point.
\item[ii)] If $f$ is non-elliptic then the only BRFP $\eta$ with stable dilation $\Lambda_\eta\leq 1$ is the Denjoy--Wolff point.
\item[iii)] If $f$ is elliptic and $\eta$ is a BRFP then $\Lambda_\eta= 1$ if and only if $\eta$ is contained in the  Gromov closure of the limit retract $\omega_f$ (Definition \ref{stronglyelliptic}).
\end{itemize}
\end{proposition}
\begin{proof}
i) By Remark \ref{feketesuper} we have that, for all $n\geq 1$, $$\log \lambda_{\eta,p}(f^n)\leq n\log\Lambda_{\eta}.$$
Let $a\in\Phi^{-1}(\eta)$, $x_0\in X$ and $n\geq 1$. By the $\delta$-Julia  Lemma \ref{deltaJ} applied to $f^n$ we have
$$h_{a,p}(f^n(x_0))\leq h_{a,p}(x_0)+\log \lambda_{\eta,p}(f^n)+4\delta\leq  h_{a,p}(x_0)
+n\log\Lambda_{\eta}+4\delta.$$
Hence $h_{a,p}(f^n(x_0))\stackrel{n\to\infty}\longrightarrow -\infty$, and by Proposition \ref{intersectsone} ii) it follows that $f^n(x_0)$ converges to $\eta\in\partial_GX$. Thus $f$ is non-elliptic, and   $\eta$ is the Denjoy-Wolff point of $f$. 

ii) Assume $f$ is non-elliptic and let $\eta$ be a BRFP with stable dilation $\leq 1$. 
Let $a\in\Phi^{-1}(\eta)$, $x_0\in X$ and $n\geq 1$. Then by the $\delta$-Julia  Lemma \ref{deltaJ} we have
$$h_{a,p}(f^n(x_0))\leq h_{a,p}(x_0)+\log \lambda_{\eta,p}(f^n)+4\delta\leq h_{a,p}(x_0)+4\delta,$$
where we used Remark \ref{feketesuper} to conclude that $\log \lambda_{\eta,p}(f^n)\leq 0$.
Hence the forward orbit $(f^n(x_0))$ is contained in the horosphere  $\{h_{a,p}\leq h_{a,p}(x_0)+4\delta\}$. By Proposition \ref{intersectsone} i) $\eta$ is the Denjoy--Wolff point of $f$.

iii) Assume  that $f$ is elliptic  and let $\eta$ be a BRFP with stable dilation $\Lambda_\eta= 1$, not contained in ${\overline {\omega_f}}^G$.  Let $a\in\Phi^{-1}(\eta)$. 
By Proposition \ref{intersectsone} ii) there exists $c\in \R$ such that 
 the   horosphere $\{h_{a,p}\leq c\}$ is contained in the open neighborhood ${\overline X}^G\setminus {\overline {\omega_f}}^G$ of $\eta$, and thus 
 $$\{h_{a,p}\leq c\}\cap \omega_f=\varnothing.$$
Let $x_0\in \{h_{a,p}\leq c-4\delta\}$. Then as above  the forward orbit $(f^n(x_0))$ is contained in the  horosphere $\{h_{a,p}\leq c\}$ but this is a contradiction since every forward orbit of $f$ admits a limit point in $\omega_f.$

Assume conversely that $\eta$ is a BRFP contained in  ${\overline {\omega_f}}^G$.
By Remark \ref{embeddingofcompactifications} $\eta$ is also a point in the Gromov boundary of $\omega_f$.
By Theorem \ref{isometryomegabig} the restriction $f|_{ \omega_f}\colon \omega_f\to  \omega_f$ is an elliptic isometry.  

Let $\gamma$ be a geodesic ray in $ \omega_f$ with endpoint $\eta$.  Then by Corollary \ref{lem:dilationgeo} we have that $$\lim_{t\to+\infty}d(\gamma(t),\gamma(0)))-d(f(\gamma(t)),\gamma(0)))\leq \log\lambda_{\eta,\gamma(0)}+4\delta,$$ where $\lambda_{\eta,\gamma(0)}$ is the dilation at $\eta$ as a BRFP of $X$.  Hence $\eta$ is a BRFP for $f|_{\omega_f}$ too. By Lemma \ref{ellipticisometry} $\eta$  is indifferent as a BRFP of $\omega_f$. Then clearly the stable dilation  of $\eta$ as a BRFP of $X$ satisfies $\log \Lambda_\eta\leq 0$. By point i) above it follows that  $\log \Lambda_\eta =0$.
\end{proof}

\begin{proposition}\label{dilationinequality}
 Let $f\colon X\to X$ be a non-elliptic non-expanding self-map of a proper geodesic Gromov hyperbolic metric space. Let $\zeta$ be the Denjoy--Wolff point of $f$ and let $\eta$ be a  BRFP different from $\zeta$. Then
 $$\log \Lambda_\eta\geq -\log\Lambda_\zeta.$$
\end{proposition}
\begin{proof}
Fix $p\in X$. By ii) of Proposition \ref{BRFPanddilation} $\eta$ is a repelling BRFP. In what follows let $n\geq 0$ be large enough such that $\log\lambda_{ \eta,p}(f^n)>0$. By Proposition \ref{prop:boundfgamma} there exists $x_n\in X$ such that
$$\log\lambda_{ \eta,p}(f^n)\geq d(x_n,f^n(x_n))-C(\delta)-1.$$
Let $a\in\Phi^{-1}(\zeta)$, by the $\delta$-Julia Lemma \ref{deltaJ} we have
$$\log\lambda_{\zeta,p}(f^n)+4\delta\geq h_{a,p}(f^n(x_n))-h_{a,p}(x_n)\geq -d(x_n,f^n(x_n)).$$
Hence
$$\log\lambda_{ \eta,p}(f^n)\geq -\log\lambda_{\zeta,p}(f^n)-4\delta-C(\delta)-1,$$
which implies 
$$\log\Lambda_{\eta}=\lim_{n\to+\infty}\frac{\log\lambda_{ \eta,p}(f^n)}{n}\geq \lim_{n\to+\infty}\frac{-\log\lambda_{\zeta,p}(f^n)-4\delta-C(\delta)-1}{n}=-\log\Lambda_{\zeta}.$$
\end{proof}

\begin{remark}
The results in this section  generalize several results for holomorphic self-maps of bounded strongly convex domains in $\C^q$. For Proposition \ref{sandman} see \cite[Proposition 4.1]{AAG}. See \cite{AbRaback} for points i) and ii) of Proposition \ref{BRFPanddilation} and for Proposition \ref{dilationinequality}. Finally see \cite[Proposition 3.4]{AbBr} for point iii) of  Proposition \ref{BRFPanddilation}.
\end{remark}

 \section{Backward orbits with bounded step converging to a repelling BRFP}\label{sectionrepelling}
We now introduce a number associated with every backward orbit with bounded step $(x_n)$, which will turn out to detect most of its dynamical behaviour (see  Propositions \ref{bpositiveiffdqg}, \ref{final2} and  Corollary \ref{final3} below). 
\begin{definition}[Backward step rate] Let $(X,d)$ be a proper geodesic Gromov hyperbolic space, and let $f\colon X\rightarrow X$ be a non-expanding map.
Let $(x_n)$ be a backward orbit with bounded step.
If $m\geq 1$, the \textit{$m$-step} of $(x_n)$ is defined as $$\sigma_m(x_n):=\lim_{n\to+\infty}d(x_n,x_{n+m}).$$
Notice that the limit exists because the sequence is not decreasing, and moreover the sequence $(\sigma_m(x_n))_m$ is subadditive.
We define the \textit{backward step rate} of $(x_n)$ as 
$$b(x_n):=\lim_{m\to+\infty}\frac{\sigma_m(x_n)}{m}=\inf_m \frac{\sigma_m(x_n)}{m}.$$
\end{definition}
\begin{remark}\label{}
Clearly $b(x_n)$ is smaller than or equal to the step $\sigma_1(x_n)$ of the backward orbit. We will see  that $b(x_n)$ carries far more dynamical information on the orbit than the step $\sigma_1(x_n)$. 
Also notice that for all $m\geq1$ we have 
$$\frac{d(x_0,x_m)}{m}\leq\frac{\sigma_m(x_m)}{m}$$
which implies
$$c(f)\leq b(x_n).$$

\end{remark}

\begin{proposition}\label{bounddilationrate} Let $(X,d)$ be a proper geodesic Gromov hyperbolic space, and let $f\colon X\rightarrow X$ be a non-expanding map.
If a point $\eta\in \partial_GX$ is a limit point of a backward orbit with bounded step $(x_n)$, then $\eta$ is a BRFP and
$$0\leq \log\Lambda_\eta\leq b(x_n).$$
\end{proposition}
\begin{proof}
Let $(x_{n_k})$ be a subsequence converging to $\eta$.
For all $m\geq 1$,
$$\sigma_m(x_n)\geq \limsup_{k\to+\infty} d(x_{n_k},p)-d(x_{{n_k}-m},p)\geq \log\lambda_{\eta,p}(f^m).$$ This shows $\log\Lambda_\eta\leq b(x_n)$. Assume we know by contradiction that $\log\Lambda_\eta<0$. Then $\eta$ is attracting, and thus $f$ is hyperbolic. Then by \cite[Proposition 6.25]{AFGG} the backward orbit $(x_n)$ has to converge to a BRFP different from the Denjoy--Wolff point $\eta$, contradiction.
\end{proof}

The goal of this section is to prove Theorem \ref{intromainone}. We will actually prove the following.
\begin{theorem}\label{montegrappa}
	\label{thm: backwardorbit}
	Let $(X,d)$ be a proper geodesic Gromov hyperbolic space, and let $f\colon X\rightarrow X$ be a non-expanding map.
	Assume that $\eta\in\partial_GX$ is a repelling BRFP with stable dilation $\Lambda_\eta>1$. Then the following holds.
	\begin{itemize}
\item[i)] There exists a backward orbit $(x_n)$ converging to $\eta$ with backward step rate $$b(x_n)=\log\Lambda_{\eta}.$$
\item[ii)] If $(x_n)$ and $(y_n)$ are two backward orbits with bounded step coverging to $\eta$, then
$$\sup_{n\geq0}d(x_n,y_n)<+\infty.$$
\item[iii)] Every backward orbit $(y_n)$ with bounded step converging to $\eta$ has backward step rate $$b(y_n)=\log\Lambda_{\eta}.$$
	\end{itemize}
\end{theorem}

\begin{lemma}
	\label{lem:orbitexits}
	Let $(X,d)$ be a proper geodesic Gromov hyperbolic space,
	let $f: X\rightarrow X$ be a non-expanding map and let $\eta\in\partial_GX$ be a repelling BRFP. Let $a\in \Phi^{-1}(\eta).$
	Then there exists $c\in \R$ so that every forward orbit of $f$ eventually avoids the horosphere $\{h_{a,p}\leq c\}$.
\end{lemma}
\begin{proof}If $f$ is non-elliptic, then  every forward orbit converges to the attracting or indifferent Denjoy--Wolff point, which is different from $\eta$. By Proposition \ref{intersectsone} i) the result holds for all $c\in \R$.
	
	Otherwise, assume that $f$ is elliptic with limit retract $\omega_f$.
	By Proposition \ref{BRFPanddilation} iii) the BRFP  $\eta$ is not contained in $\overline{\omega_f}^G$.
	 It follows from Proposition \ref{intersectsone} ii) that there exists $c\in \R$ such that 
	 $$\{h_{a,p}\leq c\}\cap \omega_f=\varnothing.$$
The result follows since every forward orbit of $f$ is eventually contained in any given neighborhood $ \omega_f\subset U\subset X$.
\end{proof}

Let now $\gamma$ be a geodesic ray with endpoint $\eta$, let $a\in \Phi^{-1}(\eta)$ be its Busemann point and let $p:=\gamma(0).$
Choose $c\in \R$ as in Lemma \ref{lem:orbitexits}.  Every forward orbit starting in the horosphere $\{h_{a,p}\leq c\}$ eventually leaves the same set.
Choose an increasing sequence $(t_k)$ in $\mathbb R$ so that $t_k \geq -c$. 
Since $h_{a, p}(\gamma(t_k)) = -t_k \leq c$, it follows that for all  $m\ge 1$ and $k\geq 0$ there exists  $n_{m,k}\ge 0$ such that for all $0\leq n\le n_{m,k}$ we have
\begin{equation}\label{finalcontradiction}
h_{a,p}( f^{mn}(\gamma(t_k)))\leq c, \quad\text{but}\quad h_{a,p}(f^{m(n_{m,k}+1)}(\gamma(t_k)))> c.
\end{equation}
Set $x_{m,k} := f^{mn_{m,k}}(\gamma(t_k))$ and $y_{m,k}:=f^{m(n_{m,k}+1)}(\gamma(t_k))$.

\begin{proposition}
	\label{prop:bounded}
	There exists $m\ge 1$ such  that  the sequences $(x_{m,k})_{k\geq 0}$ and $(y_{m,k})_{k\geq 0}$ are bounded.
\end{proposition}
\begin{proof}
	In what follows, let $m\geq 1$ be large enough such that $\log\lambda_{\eta,p}(f^m)>0$.
	 By Proposition \ref{prop:boundfgamma} applied to $f^m$ we have that
	\begin{equation}
	\label{eq:equationq1ineq}
	\limsup_{k\to+\infty}d(x_{m,k}, y_{m,k}) \le \limsup_{k\to+\infty}d(\gamma(t_k), f^m(\gamma(t_k)))\leq \log\lambda_{\eta,p}(f^m)+C(\delta).
	\end{equation}
	In particular the sequence $(d(x_{m,k}, y_{m,k}))_{k\geq 0}$ is bounded from above. Hence the sequence $(x_{m,k})_{k\geq 0}$ is bounded if and only if $(y_{m,k})_{k\geq 0}$ is bounded.
	
	Suppose by contradiction that, for every $m\ge 1$, the sequence $(x_{m,k})_{k\geq 0}$ is not bounded.
	By taking a subsequence if necessary, we may then assume that for each $m$ the sequence $(x_{m,k})$ is escaping.
	By Proposition \ref{intersectsone} i), since $h_{a,p}(x_{m,k})\leq c$ we  have that $(x_{m,k})$ converges to $\eta$.
	Therefore $(y_{m,k})$ also converges to $\eta$.
	
	By \eqref{eq:equationq1ineq}, we have that
	\begin{align*}
		\frac{\log\lambda_{\eta,p}(f^m)}{m} &\le \liminf_{k\to+\infty} \frac{1}{m}(d(p, x_{m,k}) - d(p, y_{m,k})) \le\liminf_{k\to+\infty} \frac{1}{m}d(x_{m,k}, y_{m,k})\\
		&\le \limsup_{k\to+\infty}\frac{1}{m}d(x_{m,k}, y_{m,k}) \le \frac{\log\lambda_{\eta,p}(f^m) + C(\delta)}{m},
	\end{align*}
	and that 
	\begin{align*}
	\frac{\log\lambda_{\eta,p}(f^m)}{m} &\le \liminf_{k\to+\infty} \frac{1}{m}(d(p, x_{m,k}) - d(p, y_{m,k})) \le \limsup_{k\to+\infty} \frac{1}{m}(d(p, x_{m,k}) - d_m(p, y_{m,k}))\\
	&\le \limsup_{k\to+\infty}\frac{1}{m}d(x_{m,k}, y_{m,k}) \le \frac{\log\lambda_{\eta,p}(f^m) + C(\delta)}{m}.
	\end{align*}
	The inequalities above imply the existence of a divergent sequence $(k_m)$ in $\N$ such that, if we write $\hat x_m := x_{m, k_m}$ and $\hat y_m :=y_{m, k_m}$, we have
	\begin{equation}
	\label{eq:limits}
	\lim_{m\to+\infty} \frac{d(\hat x_m, \hat y_m)}{m} = \lim_{m\to+\infty} \frac{d(p, \hat x_m) - d(p, \hat y_m)}{m}=\log\Lambda_\eta.
	\end{equation}

	Let $\gamma_m:[0, T_m]\rightarrow X$ be a geodesic segment connecting $p$ to $\hat x_m$, where $T_m := d(p, \hat x_m)$. 
	Since we are assuming that for every $m$ the sequence $(x_{m, k})_{k\geq 0}$ converges to the point $\eta$, the sequence $(k_m)$ in the previous step can be chosen so that
	\begin{equation}\label{largeenough}
	d(p, \hat x_m) \geq c+2\log\lambda_{ \eta,p}(f^m).
	\end{equation}
	In particular  for every $m$ the point $\hat y'_m:=\gamma_m(T_m-\log\lambda_{\eta,p}(f^m))$ is well defined.
	Given any $w\in X$, by Remark \ref{fourpointcondition} we have that
	\begin{equation}\label{eq:gromov_condition}
	(p, \hat x_m)_w\ge \min((p, \hat y'_m)_w, (\hat y'_m, \hat x_m)_w) -2\delta.
	\end{equation}
	
	Suppose first that the minimum in the inequality above is realized by $(p,\hat y'_m)_w$.
	Then
	$$2(p,\hat x_m)_w=d(\hat x_m,w)+d(p,w)-d(p, \hat x_m) \ge d(\hat y'_m, w) + d(p, w) - d(p, \hat y'_m) - 4\delta, $$
		hence
	\begin{equation}\label{eq:case1}
	d(\hat y'_m, w)\le d(\hat x_m,w)+ d(p,\hat y'_m)-d(p, \hat x_m)+4\delta=d(\hat x_m,w) - \log\lambda_{\eta,p}(f^m)+4\delta,
	\end{equation}
	where we used the fact that $d(p, \hat y_m') = d(p, \hat x_m) - \log\lambda_{p,\eta}(f^m)$, directly from the definition of the point $\hat y'_m$.
	
	On the other hand, suppose that the minimum in \eqref{eq:gromov_condition} is realized by $(\hat y'_m, \hat x_m)_w$.
	Then we have that
		$$2(p,\hat x_m)_w=d(\hat x_m,w)+d(p,w)-d(\hat x_m,p)\ge d(\hat y'_m, w) + d(\hat x_m, w) - d(\hat x_m, \hat y'_m) - 4\delta,$$
		hence
 \begin{equation}\label{eq:case2}d(\hat y'_m, w)- d(p, w)\le d(\hat x_m, \hat y'_m) -d(\hat x_m, p) +2\delta \le -d(\hat x_m, p) + \log\lambda_{\eta,p}(f^m)+4\delta.
 \end{equation}

	Let $(w_n)$ be a sequence in $X$ converging to  $a\in \partial^HX$.
	Then by taking a susbsequence if necessary, we may assume that either the minimum in equation \eqref{eq:gromov_condition} is always realized by $(\hat y'_m, p)_{w_n}$ or that it is always realized by $(\hat y'_m, \hat x_m)_{w_n}$.
	In the first case, for every $n$, it follows from equation \eqref{eq:case1} that
	$$
	d(\hat y'_m, w_n)-d(p, w_n)\le d(\hat x_m,w_n)-d(p, w_n) -\log\lambda_{\eta,p}(f^m)+4\delta,
	$$
	and therefore that 
	$$
	h_{a,p}(\hat y'_m)\le h_{a,p}(\hat x_m)-\log\lambda_{p,\eta}(f^m)+4\delta\leq c-\log\lambda_{\eta,p}(f^m)+4\delta.
	$$
	In the second case we have instead, using  equation \eqref{eq:case2},
	$$
	h_{a,p}(\hat y'_m)\le - d(\hat x_m, p) + \log\lambda_{\eta,p}(f^m)+4\delta.
	$$
	Since the sequence $(k_m)$ was chosen so that equation \eqref{largeenough} holds, we conclude that in both  cases the following holds:
	$$
	h_{a,p}(\hat y'_m)\leq c -\log\lambda_{\eta,p}(f^m)+4\delta.
	$$

	We claim that $\frac{1}{m}d(\hat y_m,\hat y_m')\to 0$.
	This can be proved by considering again the inequality \eqref{eq:gromov_condition} in the case $w=\hat y_m$. 
	Again, we may assume that the minimum in \eqref{eq:gromov_condition} is realized either by $(p,\hat y'_m)_{\hat y_m}$ or by $(\hat y'_m, \hat x_m)_{\hat y_m}$.
	In the first case, by equation \eqref{eq:case1} we have that
	$$ 
	d(\hat y'_m, \hat y_m) \le d(\hat x_m, \hat y_m) - \log\lambda_{\eta,p}(f^m) + 4\delta.
	$$ 
	In the second case, by equation \eqref{eq:case2}, we have instead that
	$$
	d(\hat y'_m, \hat y_m) \le d(p, \hat y_m) - d(\hat x_m, p) + \log\lambda_{\eta,p}(f^m) + 4\delta.
	$$
	By equation \eqref{eq:limits} we conclude that in both cases $\frac{1}{m}d(\hat y_m, \hat y'_m)\to 0$.

	In conclusion, we obtain that
	$$
	\frac{h_{a,p}(\hat y_m)}{m}\le \frac{h_{a,p}(\hat y'_m) + d(\hat y'_m, \hat y_m)}{m}\le \frac{c + 4\delta}{m} -\frac{\log\lambda_{\eta,p}(f^m)}{m} + \frac{d(\hat y_m', \hat y_m)}{m}.
	$$ 
	The right hand side of the inequality above converges to $-\log\Lambda_\eta<0$, which implies that whenever $m$ is sufficiently large, $h_{a,p}(\hat y_m) \leq c$, contradicting equation \eqref{finalcontradiction}.
\end{proof}

\begin{proof}[Proof of Theorem \ref{thm: backwardorbit}]
[\textit{Proof of} i)]
	The proof is similar to \cite[Theorem 2]{ArGu}, but we include it it for the convenience of the reader.
	By Proposition \ref{prop:bounded} there exists $m\ge 1$ 
	 such that the sequence $(f^{mn_{m,k}}(\gamma(t_k)))$ is bounded. Denote for simplicity 
	$n_{m,k}=n_k$. Then there exist $z_0\in X$ and a subsequence $(n_{k_{0,h}})$ such that 
	$$f^{mn_{k_{0,h}}}(\gamma(t_{k_{0,h}}))\stackrel{h\to +\infty}\longrightarrow z_0.$$
	Since $f$ is not-expanding, by Proposition \ref{prop:boundfgamma} it holds that
	$$
	d(f^{mn_{k_{0,h}}}(\gamma(t_{k_{0,h}})), f^{mn_{k_{0,h}}-1}( \gamma(t_{k_{0,h}}))) \le d(f(\gamma(t_{k_{0,h}})),\,\gamma(t_{k_{0,h}}))\leq|\log\lambda_{\eta,p}(f)|+C(\delta),
	$$
	in particular we can find a subsequence $(k_{1,h})$ of $(k_{0,h})$ and $z_1\in  X$ such that $$f^{mn_{k_{1,h}}-1}(\gamma(t_{k_{1,h}}))\stackrel{h\to +\infty}\longrightarrow z_1.$$ 
	Notice that by continuity of $f$ we have that $f(z_1)=z_0$.
	This procedure can be iterated, giving for every $\nu\ge 1$ a subsequence $(k_{\nu+1,h})$ of $(k_{\nu, h})$ such that the sequence
	$(f^{mn_{k_{\nu,h}}-\nu}(\gamma(t_{k_{\nu,h}})))$ converges to a point $z_\nu\in X$ such that $f(z_{\nu})= z_{\nu-1}$. 
	 Furthermore, again by Proposition \ref{prop:boundfgamma}, we have for all $\mu\geq 1$, 
	\begin{align*}
	d(z_{\nu}, z_{\nu-\mu})&=\lim_{h\to+\infty}d(f^{mn_{k_{\nu,h}}-\nu}(\gamma(t_{k_{\nu,h}})), f^{mn_{k_{\nu,h}}-\nu-\mu}(\gamma(t_{k_{\nu,h}})))\\
	&\le\lim_{h\to+\infty} d(f^\mu(\gamma(t_{k_{\nu,h}})), \gamma(t_{k_{\nu,h}})) \leq|\log\lambda_{\eta,p}(f^\mu)|+C(\delta),
	\end{align*}
	and therefore  $b(z_n)\leq \log\Lambda_\eta$. 
	
	It remains to show that the backward orbit  $(z_\nu)$ converges to $\eta$. It is enough to show that the subsequence 
	$(z_{m\nu})$ converges to $\eta$.
	Notice that by construction for all $\nu\ge 0$ we have that the point $z_{m \nu}$ belongs to the horosphere $\{h_{a,p}\leq c\}$.
	By Proposition \ref{intersectsone} i), either $z_{m \nu}\to \eta$ or there exists a subsequence $z_{m \nu_k}\to z'\in \{h_{a,p}\leq c\}$. 
	In the second case for every $i\in\mathbb N$ we have that
	$$
	f^{mi}(z')=\lim_{k\to\infty}f^{mi}(z_{m\nu_k})=\lim_{k\to\infty}z_{m(\nu_k-i)}\in\{h_{a,p}\leq c\}.
	$$
	We conclude that there exists a subsequence of the forward orbit of  $z'$ contained in the horosphere $\{h_{a,p}\leq c\}$, which is not possible thanks to the choice of $c$ (see Lemma \ref{lem:orbitexits}). 
	
	
	By Proposition \ref{bounddilationrate} it follows that the backward step rate  $b(z_\nu)$ is bounded from below by $\log\Lambda_\eta$, and therefore  it must be equal to $\log\Lambda_\eta$.

[\textit{Proof of} ii)] 
Let $(x_n)$ and $(y_n)$ be backward orbits with bounded step converging to $\eta$.
The backward orbits  are discrete quasi-geodesics, and thus can be interpolated with quasi-geodesic rays. Hence by  Gromov's shadowing lemma (Theorem \ref{gromovshadowing}) we have that there exists $M\geq 0$ such that
$$\sup_{n\geq0}\inf_{m\geq0}d(x_n,y_m)\leq M.$$
Consider the complete orbits of $(x_n)$ and $(z_n)$, setting $x_{-n}:=f^{n}(x_0)$ and $y_{-n}:=f^{n}(y_0)$ for all $n>0$.
The sequences $(x_n)$ and $(y_n)$ converge to $\eta$ as $n\to+\infty$. When $n\to -\infty$, then if $f$ is non-elliptic they  converge to the Denjoy--Wolff point of $f$, which is different from $\eta$, while if $f$ is elliptic they accumulate on the limit retract $\omega_f$, which by Proposition \ref{BRFPanddilation} iii)  does not contain $\eta$ in its Gromov closure.
Hence in both cases there exists $N\geq 0$ such that
$d(x_N,y_m)> M$ for all $m< 0$. Moreover, there exists $L\geq 0$ such that $d(x_N,y_m)> M$ for all $m> L$.

For all $n\geq N$ let  $m_n\in\N$ be such that $d(x_n,y_{m_n})\leq M$. By the non-expansivity of $f$,
$$d(x_N,y_{m_n+N-n})\leq M,$$
which implies
$$0\leq m_n+N-n\leq L.$$
In particular $$|n-m_n|\leq L+N.$$
Finally, for all $n\geq N$,
$$d(x_n,y_n)\leq d(x_n,y_{m_n})+d(y_{m_n},y_n)\leq M+\sigma_1|n-m_n|\leq M+\sigma_1(L+N).$$

[\textit{Proof of} iii)] 
Let $(y_n)$ be a backward orbit with bounded step converging to $\eta$, and let $(x_n)$ be the backward orbit given by point i).
By ii) there exists $M\geq 0$ such that $d(x_n,y_n)\leq M$ for all $n\geq0$.  
Hence for all $m\geq 0$, $$|\sigma_m(x_n)-\sigma_m(y_n)|\leq 2M,$$
which implies $b(y_n)=b(x_n)=\log\Lambda_\eta.$
\end{proof}

We conclude this section showing that point i) and ii) of Theorem \ref{montegrappa} are not true in general if the point $\eta$ is indifferent.
	\begin{example}
	Consider the metric space $(\R_{>0},d)$ with $d(x,y)=|\ln\frac{x}{y}|$, and let $f\colon \R_{>0}\to \R_{>0}$ be the non-expanding map $f(t)=t+1$. Then $\partial_G\R_{>0}=\{0,+\infty\}$, the indifferent BRFP $+\infty$ is the Denjoy--Wolff point  of $f$ but there are not backward orbits converging to $+\infty$. A similar example in the holomorphic setting is given by the self-map $z\to z+1$ in the right half-plane $\H$ endowed with the Poincar\'e distance.
\end{example}
\begin{example}
Let $X=\C\setminus \R_{\leq 0}$ endowed with the Poincar\'e distance, and let $f\colon X\to X$ be defined by $f(z)=z+1$. Then the Denjoy--Wolff point $\infty$ is indifferent. The two backward orbits  with bounded step $x_n:= (-n,1)$ and $y_n:=(-n,-1)$ converge to the Denjoy--Wolff point, but $d(x_n,y_n)\to+\infty.$
\end{example}
Point iii) of Theorem \ref{montegrappa} actually holds also if $\eta$ is indifferent, as will be shown in the next section.
\section{Backward Denjoy--Wolff theorem}\label{backwardDW}
In this section we prove Theorem \ref{intromaintwo}. The condition that backward orbit has bounded step is crucial, as the following example shows.
\begin{example}[Backward orbits with unbounded step]
	Let $X=\mathbb{S}^1\times (0,+\infty)$ with the Riemannian metric
	$$g(\alpha,t)=\frac{(d\alpha)^2+(dt)^2}{t^2},$$
	and $d$ the corresponding distance.
The metric space $(X,d)$ is proper geodesic Gromov hyperbolic because is isometric to the punctured disk $\D^*$ with the hyperbolic distance (for the hyperbolicity of $\D^*$ see for example Lemma 5.4 in \cite{RodriTou}) and its Gromov boundary is canonically homeomorphic to disjoint union  between $\mathbb{S}^1$ and $+\infty$.
	Let $\vartheta\in [0,2\pi)\backslash\pi\Q$ and let  $R_\vartheta\colon \mathbb{S}^1\to \mathbb{S}^1$ be the counterclockwise rotation by an angle $\vartheta$. Then the map
	$$f((\alpha,t))=(R_\vartheta(\alpha),2t)$$ is a non-expanding hyperbolic map (with Denjoy-Wolff point $+\infty$).
	Fix $x_0:=(\alpha,t)\in X$, then the sequence $x_n:=(R_{n\vartheta}^{-1}(\alpha),2^{-n}t) $ is a backward orbit with $\mathbb{S}^1$ as limit set.
\end{example}


We start describing the only backward orbits which are not escaping.
 \begin{proposition}\label{calkabackward}
 Let $X$ be a proper  metric space and let $f\colon X\to X$ be a non-expanding self-map. 
 If a backward orbit $(x_n)$ is not escaping, then the map $f$ is elliptic, and $(x_n)$ is a relatively compact orbit of the form $(x_n)=(f|_{\omega_f}^{-n}(x_0)).$
 \end{proposition}
 \begin{proof}
 Assume that the backward orbit $(x_n)$ is not escaping. Then there exists a subsequence $(x_{n_k})$ 
 converging to a point $w_0$ in $X$.  Then we have that
 $$d(x_0, f^{n_k}(w_0))= d(f^{n_k}(x_{n_k}),  f^{n_k}(w_0))\leq d(x_{n_k},w_0)\to 0.$$
 Hence $f$ is elliptic and the point $x_0$ belongs to the limit set of the forward orbit of $w_0$. Hence  $x_0\in\omega_f$. Similarly we obtain $x_n\in\omega_f$ for all $n\geq 1$.
 \end{proof}


The proof of Theorem \ref{intromaintwo} is split in the two following results.
\begin{proposition}\label{bpositiveiffdqg}\label{backwardhyperbolic}
Let $X$ be a proper geodesic Gromov hyperbolic metric space and let $f\colon X\to X$ be a non-expanding self-map. Let $(x_n)$ be an escaping backward orbit with bounded step. 
Then the following are equivalent:
\begin{enumerate}
\item $b(x_n)>0$;
\item $(x_n)$ is a discrete quasigeodesic;
\item $(x_n)$ converges to a BRFP $\eta$ inside a geodesic region with vertex $\eta$;
\item $(x_n)$ converges to a BRFP $\eta$ and
$$\lim_{n\to+\infty}h_{a,p}(x_n)=-\infty,\quad\forall a\in \Phi^{-1}(\eta),$$
 that is, $(x_n)$ is eventually contained in every horosphere centered in a point of $\Phi^{-1}(\eta)$;
\item $(x_n)$ converges to a BRFP $\eta$ and 
$$\liminf_{n\to+\infty}h_{a,p}(x_n)=-\infty,\quad\forall a\in \Phi^{-1}(\eta);$$
\item  $(x_n)$ converges to a repelling BRFP;
\item  $(x_n)$ converges to a repelling BRFP $\eta$ and $b(x_n)=\log\Lambda_\eta$.
\end{enumerate}
\end{proposition}
\begin{proof}
$(1)\Rightarrow (2)$ 
Interpolate $(x_n)$ with a curve $\gamma:[0,+\infty)\rightarrow X$ defined as
$$\gamma(t)=\gamma_{\lfloor t\rfloor}(t-\lfloor t\rfloor), \ \ \ t>0, $$
where $\gamma_n:[0,1]\rightarrow X$ is a curve from $x_n$ to $x_{n+1}$ with length equal to $d(x_n,x_{n+1})$. 
Since $b(x_n)=\inf_n\frac{\sigma_n(x_n)}{n}$, it follows that  $\sigma_n\geq b(x_n)n$ for all $n>0$. Let $A\geq 1$ be such that $b(x_n)\geq A^{-1},$ and let  $N\geq 1$. Then there exists $M_N\geq 0$ such that for all $m\geq M_N$ and for all $n=1,\dots,N+1$ we have
$$d(x_m,x_{m+n})\geq A^{-1}n.$$ 
We want to prove that for all $N\geq 1$ the curve $\gamma|_{[M_N,+\infty)}$ is a $N$-local $(A\sigma_1,2A\sigma_1^2)$-quasi-geodesic
in the sense of \cite[Definition 1.1]{CDP}
i.e. for all $M_N\leq s\leq t$ with $|t-s|\leq N$  we have
$$\ell(\gamma|_{[s,t]})\leq A\sigma_1 d(\gamma(s),\gamma(t))+2A\sigma_1^2. $$
Indeed, 
\begin{align*}
\ell(\gamma|_{[s,t]})&\leq\sum_{k=\lfloor s\rfloor}^{\lfloor t\rfloor}\ell(\gamma_k)=\sum_{k=\lfloor s\rfloor}^{\lfloor t\rfloor}d(x_k,x_{k+1})\leq \sigma_1(\lfloor t\rfloor+1-\lfloor s\rfloor)\\
&\leq A\sigma_1d(x_{\lfloor s\rfloor},x_{\lfloor t\rfloor+1})\leq A\sigma_1d(\gamma(s),\gamma(t))+2A\sigma_1^2.
\end{align*}

Now by \cite[Theorem 1.4]{CDP} there exist $\hat A\geq 1,\hat B\geq 0$ and $N=N(A,\sigma_1,\delta)\geq 1$ such that $\gamma|_{[M_N,+\infty)}$ is a (global) $(\hat{A},\hat{B})$-quasi-geodesic in the sense of \cite{CDP}, that is for all $M_N\leq s<t$
$$\ell(\gamma|_{[s,t]})\leq \hat A d(\gamma(s),\gamma(t))+\hat B.$$
In particular, for all $n\geq M_N$ and $m\geq1$ we have
$$d(x_0,x_1)m\leq\sum_{k=n}^{n+m-1}d(x_k,x_{k+1})=\sum_{k=n}^{n+m-1}\ell(\gamma|_{[k,k+1]})=\ell(\gamma|_{[n,n+m]})\leq\hat Ad(x_n,x_{n+m})+\hat B$$
so 
$$d(x_n,x_{n+m})\geq \hat A^{-1}d(x_0,x_1) m-\hat A^{-1}\hat B.$$
Finally, 
$$d(x_n,x_{n+m})\leq\sum_{k=n}^{n+m-1}d(x_k,x_{k+1})\leq \sigma_1 m$$
so $(x_n)_{n\geq M_N}$ is a discrete quasi-geodesic.
%
%

$(2)\Rightarrow (3)$ 
The discrete quasi-geodesic $(x_n)$ can be interpolated with a quasi-geodesic ray. Then the result  follows from Gromov's shadowing lemma (Theorem \ref{gromovshadowing}).

$(3)\Rightarrow (4)$
Follows from Proposition \ref{intersectsone} iii).

$(4)\Rightarrow (5)$ Trivial.
%

$(5)\Rightarrow (6)$ 
by Proposition \ref{bounddilationrate} the BRFP $\eta$ cannot be attracting.
Assume by contradiction that  $\eta$ is indifferent. Let $a\in \Phi^{-1}(\eta)$, and let $(x_{n_k})$ be a subsequence such that 
$$\lim_{k\to+\infty} h_{a,p}(x_{n_k})=-\infty.$$
Then by the $\delta$-Julia  Lemma \ref{deltaJ} we have
$$h_{a,p}(x_0)\leq h_{a,p}(x_{n_k})+4\delta\stackrel{n\to+\infty}\longrightarrow-\infty,$$
which is a contradiction.

$(6)\Rightarrow (7)$ Follows from Theorem \ref{montegrappa} iii). 

$(7)\Rightarrow (1)$ Trivial.

\end{proof}

\begin{proposition}\label{final2}
Let $X$ be a proper geodesic Gromov hyperbolic metric space and let $f\colon X\to X$ be a non-expanding map. Let $(x_n)$ be an escaping backward orbit with bounded step and assume that  $b(x_n)=0$. Then $f$ is either weakly elliptic or parabolic. In the parabolic case,  the orbit   $(x_n)$ converges to the indifferent Denjoy--Wolff point $\zeta$ of $f$, and
$$\liminf_{n\to+\infty} h_{a,p}(x_n)>-\infty, \quad \forall a\in \Phi^{-1}(\eta),$$
 that is, there exists a horosphere centered in a point of $\Phi^{-1}(\zeta)$ which does not contain any point of $(x_n)$.
\end{proposition}
\proof
By Proposition \ref{bounddilationrate}  the limits points of $(x_n)$ are  indifferent. Then the map $f$ cannot be hyperbolic. The map $f$ cannot be strongly elliptic either, indeed from (3) of Proposition \ref{BRFPanddilation} all the BRFP are repelling.  
Hence  $f$ is either weakly elliptic or parabolic.
If $f$ is parabolic, then the backward orbit $(x_n)$ has to converge to the Denjoy--Wolff point of $f$. The last statement follows from Proposition \ref{bpositiveiffdqg}.
\endproof

\begin{corollary}\label{final3}
Let $X$ be a proper geodesic Gromov hyperbolic metric space. Let $f\colon X\to X$ be a non-expanding map. Let $(x_n)$ be an escaping backward orbit with bounded step. Then the limit
$$c(x_n):=\lim_{n\to+\infty} \frac{d(x_0,x_n)}{n}$$ exists and equals the backward step rate $b(x_n).$
\end{corollary}
\proof
Clearly for all $n\geq0$ 
$$d(x_0,x_n)\leq\sigma_n ,$$
so $$\limsup_{n\to\infty}\frac{d(x_0,x_n)}{n}\leq b(x_n).$$ It follows that if $b(x_n)=0$ the limit exists and is equal to $0$.
 If instead $b(x_n)>0$, then  Theorem \ref{backwardhyperbolic} $(x_n)$ converges to a repelling point $\eta\in\partial_GX$ with stable dilation $\log\Lambda_{\eta}=b(x_n)>0$. 
Let $0<a< \log \Lambda_\eta$. Then there exists $m\geq 1$ such that 
$$\frac{\log\lambda_{\eta,p}(f^m)}{m}>a.$$
It follows that there exists $N\geq 0$ such that for all $n\geq N$ we have
$$d(p,x_{m(n+1)})-d(p,x_{mn})\geq am.$$
Hence, for all $n,k\geq N$,  $$d(x_{mn}, x_{mk})\geq am|n-k|,$$ which gives $$\liminf_{n\to\infty}\frac{d(x_0,x_{mn})}{n}\geq am.$$
Now for all $n\geq0$ there exists $k\geq0$ such that $mk\leq n<m(k+1)$, so
$$d(x_0,x_n)\geq d(x_0,x_{mk})-d(x_{mk},x_n)\geq d(x_0,x_{mk})-(n-mk)\sigma_1\geq d(x_0,x_{mk})-m\sigma_1,$$
which implies 
$$\liminf_{n\to+\infty}\frac{d(x_0,x_n)}{n}\geq\liminf_{k\to+\infty}\left[\frac{d(x_0,x_{mk})}{mk}-\frac{m\sigma_1}{mk}\right]\geq a$$
for all $0<a< \log \Lambda_\eta$, hence
$$\liminf_{n\to+\infty}\frac{d(x_0,x_n)}{n}\geq\log\Lambda_{\eta}=b(x_n).$$
\endproof

We leave the following open question.
\begin{question}
Let $(X,d)$ a Gromov hyperbolic metric space and let $f\colon X\to X$ be a weakly elliptic non-expanding map. Can there exist an escaping backward orbit with bounded step $(x_n)$ not converging to a point of the Gromov boundary? Clearly such an orbit would satisfy $b(x_n)=0$ and thus its limit set would be contained in the Gromov closure of the limit retract $\omega_f$. 
\end{question}

We conclude giving an example of a weakly elliptic non-expanding map with a backward orbit with bounded step converging to a point in the Gromov closure of the limit retract.

\begin{example}	Let $\H_u\subset \C$ be the upper half-plane  endowed with the Poincar\'e distance. Let $f:\H_u\longrightarrow \H_u$ be defined  by $f(x+iy)=[x-1]_+-[-x-1]_++iy$ where $[\cdot]_+$ is the positive part. The map $f$ is weakly elliptic non-expanding with limit retract the geodesic line between $0$ and $\infty$. The associated retraction is $g(x+iy)=iy$. Finally, the sequence $x_n=n+i$ is a backward orbit with bounded step converging to $+\infty$.
\end{example}

\end{document}